\documentclass[12pt]{amsart}
\usepackage{amssymb}
\title{Derived equivalences for Rational Cherednik algebras}
\author{Ivan Losev}
\newcommand{\h}{\mathfrak{h}}
\newcommand{\Hom}{\operatorname{Hom}}
\newcommand{\Ext}{\operatorname{Ext}}
\newcommand{\B}{\mathcal{B}}
\newcommand{\OCat}{\mathcal{O}}
\newcommand{\End}{\operatorname{End}}
\newcommand{\param}{\mathfrak{p}}
\newcommand{\C}{\mathbb{C}}
\newcommand{\Z}{\mathbb{Z}}
\newcommand{\Q}{\mathbb{Q}}
\newcommand{\KZ}{\operatorname{KZ}}
\newcommand{\Res}{\operatorname{Res}}
\newcommand{\Ind}{\operatorname{Ind}}
\newcommand{\Cat}{\mathcal{C}}
\newcommand{\gr}{\operatorname{gr}}
\newcommand{\Irr}{\operatorname{Irr}}
\newcommand{\loc}{\mathsf{loc}}
\newcommand{\LS}{\operatorname{LS}}
\newcommand{\HC}{\operatorname{HC}}
\newcommand{\VA}{\operatorname{V}}

\newcommand{\Leaf}{\mathcal{L}}
\newcommand{\A}{\mathcal{A}}
\newcommand{\M}{\mathcal{M}}
\newcommand{\J}{\mathcal{J}}

\newcommand{\quo}{/\!/}
\newcommand{\g}{\mathfrak{g}}
\newtheorem{Thm}{Theorem}[section]
\newtheorem{Prop}[Thm]{Proposition}
\newtheorem{Cor}[Thm]{Corollary}
\newtheorem{Lem}[Thm]{Lemma}
\theoremstyle{definition}

\newtheorem{Rem}[Thm]{Remark}

\numberwithin{equation}{section}
\address{Department
of Mathematics, Northeastern University, Boston MA 02115 USA}
\email{i.loseu@neu.edu}
\begin{document}
\begin{abstract}
Let $W$ be a complex reflection group and $H_c(W)$ the Rational Cherednik algebra for $W$ depending on
a parameter $c$. One can consider the category $\OCat$ for $H_c(W)$.
We prove a conjecture of Rouquier that the categories $\mathcal{O}$ for $H_c(W)$ and $H_{c'}(W)$
are derived equivalent provided the parameters $c,c'$ have integral difference.
Two main ingredients of the proof are a connection between the Ringel
duality and Harish-Chandra bimodules and an analog of a  deformation technique developed by the author and
Bezrukavnikov. We also show that some of the derived equivalences we construct are perverse.
\end{abstract}
\maketitle
\section{Introduction}
\subsection{Hecke algebras}
Let $W$ be the Weyl group of some connected reductive group $G$.  By $S$ we denote the set of reflections in $W$.
The group algebra $\C W$ admits a classical
deformation, the Hecke algebra $\mathcal{H}_{q}(W)$, where $q\in (\C^\times)^{S/W}$. 
The representation theory of $\mathcal{H}_q(W)$ is most interesting when $q$ has finite (and sufficiently small) order. In this case this representation theory is similar to
(but easier than) the modular representation theory of the group $W$.

The Hecke algebras still make sense when $W$ is a complex reflection group, see \cite{BMR}, we will
recall the definition below. Those are still algebras $\mathcal{H}_q(W)$ with $q\in (\C^\times)^{S/W}$.
Their structure is more complicated than in the Weyl
group case, for example, it is not known in the full generality whether $\dim \mathcal{H}_q(W)=|W|$.
However, the algebra $\mathcal{H}_q(W)$ always has the maximal finite dimensional quotient $\underline{\mathcal{H}}_q(W)$ and the dimension of this quotient is $|W|$, \cite{Hecke_dim}.
 When $q$ is Zariski generic, we have $\underline{\mathcal{H}}_q(W)\cong \C W$.

Let us point out that the algebra $\underline{\mathcal{H}}_q(W)$, in general, has infinite homological dimension
so is ``singular''. So one can ask about a ``resolution of singularities''. Such resolutions
are provided by categories $\mathcal{O}$ for Rational Cherednik algebras to be described briefly
in the next subsection.

\subsection{Cherednik algebras and their categories $\OCat$}
Let $\h$ denote the reflection representation of a complex reflection group $W$. A Rational Cherednik
algebra $H_c(W)$ is a flat deformation of the skew-group ring $S(\h\oplus \h^*)\#W$ depending
on a parameter $c\in \param:=\C^{S/W}$. We write $H_c$ instead of $H_c(W)$ if this does not
create ambiguity.

This algebra admits a triangular decomposition $H_c=S(\h^*)\otimes \C W\otimes S(\h)$
(as a vector space), where  $S(\h^*), \C W, S(\h)$ are embedded as subalgebras. So it makes
sense to consider a {\it category} $\OCat$. This is a full subcategory in $H_c\operatorname{-mod}$
consisting of all modules that are finitely generated over $S(\h^*)$ and have locally nilpotent
action of $\h$. Let us denote the category $\OCat$ by $\OCat_c$ or by $\OCat_c(W)$.
It has analogs of Verma modules, $\Delta_c(\lambda)$, parameterized by the irreducible representations $\lambda$ of $W$. Each
$\Delta_c(\lambda)$ has a unique irreducible quotient, $L_c(\lambda)$, and the assignment
$\lambda\mapsto L_c(\lambda)$ is a bijection between the sets of the irreducible $W$-modules
and the set of the irreducible objects in $\OCat_c$.

The category $\mathcal{O}_c$ has a so called highest weight structure that axiomatizes certain
upper triangularity properties similar to those of the BGG categories $\OCat$. 
We will recall a precise definition
later. One consequence of being highest weight is that $\OCat_c$ has finite homological dimension.

Moreover, there is a quotient functor $\operatorname{KZ}_c:\OCat_c\twoheadrightarrow \underline{\mathcal{H}}_q\operatorname{-mod}$ introduced in \cite{GGOR}
that is fully faithful on the projective objects (a highest weight cover in the terminology of Rouquier,
\cite[Section 4.2]{rouqqsch}). So we can view $\OCat_c$ as a ``resolution of singularities''
for $\underline{\mathcal{H}}_q\operatorname{-mod}$. Here $q$ is recovered from $c$ by some kind of exponentiation: there is a $\Z$-lattice $\param_\Z\subset \param$ such that the set of Hecke parameters is identified with
$\param/\param_\Z$ and $q=c+\param_\Z$.

\subsection{Derived equivalences}
Now let $c,c'$ be two Cherednik parameters with $c-c'\in \param_\Z$ so that $\OCat_c,\OCat_{c'}$
are two resolutions of singularities for $\underline{\mathcal{H}}_q\operatorname{-mod}$. A natural question
to ask is whether these two resolutions are derived equivalent. Rouquier conjectured that this is
so in \cite[Conjecture 5.6]{rouqqsch}. The main goal of this paper is to prove this conjecture.

\begin{Thm}\label{Thm:der_equiv}
Let $c,c'\in \param$ and $c'-c\in \param_\Z$. Then there is a derived equivalence  $D^b(\OCat_c)\xrightarrow{\sim}
D^b(\OCat_{c'})$ intertwining the functors $\KZ_c,\KZ_{c'}$.
\end{Thm}

%

Theorem \ref{Thm:der_equiv} was known  for $W=G(\ell,1,n)$. Recall
that this group is realized as $\mathfrak{S}_n\ltimes \mu_\ell^n$, where $\mu_\ell$ denotes
the group of $\ell$th roots of $1$, and its reflection representation is $\C^n$. In this case,
Theorem \ref{Thm:der_equiv} was proved in \cite[Section 5]{GL} and is  a consequence of the quantized
derived McKay equivalence.  Peculiarly, the proof is based on the study of  actual algebro-geometric resolutions of $(\h\oplus \h^*)/W$.

\subsection{Perverse equivalences}\label{SS_perv_intro}
We will also prove in Section \ref{S_perv} that some of the equivalences in Theorem \ref{Thm:der_equiv}
are perverse in the sense of \cite[2.6]{rouquier_ICM}. Some special cases of this were established
in \cite[Section 7]{BL}.

Let us recall the definition of a perverse equivalence. Let $\Cat^1,\Cat^2$ be two abelian categories.
Suppose $\Cat^i, i=1,2,$ is equipped with a filtration $\Cat^i=\Cat^i_0\supset \Cat^i_1\supset\Cat^i_2\supset\ldots\supset \Cat^i_q\supset \Cat^i_{q+1}=\{0\}$
by Serre subcategories. We say that an equivalence $\varphi:D^b(\Cat^1)\rightarrow D^b(\Cat^2)$
of triangulated categories is {\it perverse} with respect to the filtrations if the following holds:
\begin{itemize}
\item[(I)] $\varphi$ restricts to an equivalence of $D^b_{\Cat^1_j}(\Cat^1)$ and $D^b_{\Cat^2_j}(\Cat^2)$ for all $j=1,\ldots,q$.
Here we write  $D^b_{\Cat^i_j}(\Cat^i)$ for the full subcategory of $D^b(\Cat^i)$ consisting of all complexes
with homology  in $\Cat^i_j$.
\item[(II)] For $M\in \Cat^1_j$ we have $H_k(\varphi M)=0$ for $k<j$.
\item[(III)] The functor $M\mapsto H_j(\varphi M)$ induces an equivalence $\Cat^1_j/\Cat^1_{j+1}\xrightarrow{\sim}
\Cat^2_j/\Cat^2_{j+1}$. Moreover,  $H_k(\varphi M)\in \Cat^2_{j+1}$ for $k>j$.
\end{itemize}

The definition of filtrations on $\Cat^1:=\OCat_c, \Cat^2:=\OCat_{c'}$ making some of equivalences in
Theorem \ref{Thm:der_equiv} perverse is technical. Roughly speaking these filtrations are obtained by degenerating
the filtration by dimension of support.

\subsection{Ideas of the proof and the content}\label{SS_ideas}
Our key idea is the same as in the proof of \cite[Theorem 7.2]{BL}: we want to prove Theorem \ref{Thm:der_equiv}
at a Weil generic point of a hyperplane (if $\OCat_c$ is not semisimple, then $c$ lies in a countable union
of hyperplanes) and then to degenerate to a special point. Recall that by a Weil generic point of an algebraic
variety over $\C$ one means a point lying outside the countable union of algebraic subvarieties.
At a Weil generic point  our derived equivalence will be the Ringel duality.
However, in order to degenerate we will need to realize this functor as a product with a Harish-Chandra
bimodule. We will see that the bimodules of interest form a ``family'' and this will allow us
to degenerate.

Let us now describe the content of the paper. In Section \ref{S_prelim_O} we will recall some classical
facts about Hecke algebras, Cherednik algebras, their categories $\OCat$, KZ functors, Ringel duality,
induction and restriction functors, this section contains no new results. In Section \ref{S_prelim_HC}
we recall Harish-Chandra bimodules, the restriction functors for those and various properties
of Tor's and Ext's involving those bimodules.

In Section \ref{S_Ringel_HC}
we realize the (inverse covariant) Ringel duality as the derived tensor product with a suitable
Harish-Chandra bimodule. In Section \ref{S_derived} we prove Theorem \ref{Thm:der_equiv}.
Section \ref{S_perv} deals with perverse equivalences. Finally, in Section \ref{S_open} we state
two open problems.

{\bf Acknowledgements}. I would like to thank Yuri Berest, Roman Bezrukavnikov,
Pavel Etingof, Raphael Rouquier and Jose Simental for stimulating
discussions. I am also grateful to the referees whose comments allowed me to improve the exposition. This work was supported by the NSF  under Grant  DMS-1161584.

\section{Categories O}\label{S_prelim_O}
\subsection{Hecke algebras}
Let $W$ be a complex reflection group and $\h$ be its reflection representation. For a reflection hyperplane
$H$, the pointwise stabilizer $W_H$ is cyclic, let $\ell_H$ be the order of this group. The set of
the reflection hyperplanes will be denoted by $\mathfrak{H}$. Set $\h^{reg}:=\{x\in \h| W_x=\{1\}\}=\h\setminus
\bigcup_{H\in \mathfrak{H}}H$.

Consider the {\it braid group} $B_W$ that is, by definition, $\pi_1(\h^{reg}/W)$. The group is generated by elements
$T_H$, one for each reflection hyperplane, where, roughly speaking, $T_H$ is a loop in $\h^{reg}/W$ given
given by a rotation around $H$ by the angle $2\pi/\ell_H$. The structure of the braid groups was studied in more
detail in \cite{BMR}.

Now let us define the Hecke algebras for $W$. To each conjugacy class of reflection hyperplanes  $H$
we assign nonzero complex numbers $q_{H,1},\ldots, q_{H,\ell_H}$.
We denote the collection of $q_{H,i}$ by $q$. By definition, $\mathcal{H}_q(=\mathcal{H}_q(W))$ is the quotient
of $\C B_W$ by the relations $\prod_{i=1}^{\ell_H} (T_H-q_{H,i})$, one for each reflection hyperplane
$H$. For example, if we put $q_{H,i}=\exp(2\pi\sqrt{-1}i/\ell_H)$, then we get $\C W$.  We remark that rescaling the parameters $q_{H,i}, i=1,\ldots,\ell_H,$ by a common factor (one for each conjugacy class in $\mathfrak{H}/W$) gives rise to an isomorphic algebra, see \cite[3.3.3]{rouqqsch}. So the number of parameters for the Hecke algebra is
actually $|S/W|$, where we write $S$ for the set of complex reflections in $W$.

Let us consider two important families of examples. First, assume $W$ is a real reflection group so
that all $\ell_H$ are equal to $2$.  In the corresponding
Coxeter diagram $I$, let $m_{ij}$ be the multiplicity of the edge between vertices $i,j$. The braid group
$B_W$ is generated by elements $T_i, i\in I,$ subject to the relation $T_iT_jT_iT_j\ldots =T_j T_i T_j T_i\ldots$,
where on both sides we have $m_{ij}+2$ factors. If $W$ is irreducible, then the number of parameters $q$ is either one or two and the additional
relations for the Hecke algebra are $(T_i-q_i)(T_i+1)=0$ (we have $q_i=q_j$ if the reflections $s_i,s_j$
in $W$ are conjugate).

The second family is for the groups $W=G(\ell,1,n)$. When $\ell=1,2$, we get Weyl groups of type $A_{n-1},B_n$,
respectively.
For $\ell>1$,  the braid group $B_W$ is the affine braid group of type $A$, it is generated by elements $T_0,\ldots,T_{n-1}$
subject to the relations $T_0T_1T_0T_1=T_1T_0T_1T_0, T_iT_{i+1}T_i=T_{i+1}T_i T_{i+1}$ for $1\leqslant i\leqslant n-2$
and $T_iT_j=T_jT_i$ for $|i-j|>1$. The Hecke algebra is the quotient of  $\C B_W$  by
the relations $(T_i+1)(T_i-q)=0$ for $i>0$
and $\prod_{j=0}^{\ell-1}(T_0-Q_j)=0$. Here $Q_0,\ldots,Q_{\ell-1}$ are nonzero complex numbers, we can take
$Q_0=1$ without changing the algebra.

Let us point out that we can define the Hecke algebra $\mathcal{H}_{R,q}$ over any commutative
ring $R$ (the entries of $q$ are supposed to be invertible elements of $R$).

It was shown in \cite{Hecke_dim} that the algebra $\mathcal{H}_q(W)$ admits a maximal finite dimensional quotient
to be denoted by $\underline{\mathcal{H}}_q(W)$ whose dimension equals $|W|$. These algebras form a flat family
over $(\C^\times)^{S/W}$.   So for a $\C[(\C^\times)^{S/W}]$-algebra $R$, we have an algebra $\underline{\mathcal{H}}_{R,q}(W)$ that is a projective
$R$-module of rank $|W|$.



\subsection{Rational Cherednik algebras}
Recall that we have chosen a $W_H$-eigenvector $\alpha_H^\vee\in \h$ with nontrivial eigencharacter.
Let $\alpha^\vee_H$ denote an eigenvector for $W_H$ in $\h$ with a non-trivial eigen-character so that $H\oplus \C\alpha_H^\vee=\h$.
Pick a $W_H$-eigenvector $\alpha_H\in \h^*$ with $\langle \alpha_H,\alpha_H^\vee\rangle=2$.
For a complex reflection $s$ we write $\alpha_s,\alpha_s^\vee$ for $\alpha_H,\alpha_H^\vee$
where $H=\h^s$.   Let $c:S\rightarrow \C$
be a function constant on the conjugacy classes. The space of such functions is denoted
by $\param$, it is a vector space of dimension $|S/W|$.

By definition, \cite[Section 1.4]{EG}, \cite[Section 3.1]{GGOR},
the Rational Cherednik algebra  $H_c(=H_c(W)=H_c(W,\h))$ is the quotient of $T(\h\oplus \h^*)\# W$
by the following relations:
$$[x,x']=[y,y']=0,  [y,x]=\langle y,x\rangle-\sum_{s\in S}c(s)\langle x,\alpha_s^\vee\rangle\langle y,\alpha_s\rangle s, \quad x,x'\in \h^*, y,y'\in \h.$$
We would like to point out that $H_c$ is the specialization to
$c$  of a $\C[\param]$-algebra $H_{\param}$ defined as follows. The
space $\param^*$ has basis ${\bf c}_s$ naturally numbered
by the conjugacy classes of reflections. Then $H_{\param}$ is the quotient of $T(\h\oplus \h^*)\# W\otimes \C[\param]$
by the relations similar to the above but with $c(s)\in \C$ replaced with ${\bf c}_s\in \param^*$.
For a commutative algebra $R$ with a $W$-invariant map $c:S\rightarrow R$ (equivalently, with an algebra 
homomorphism $\C[\param]\rightarrow R$) we can consider the algebra $H_{R,c}=R\otimes_{\C[\param]}H_{\param}$. If $R=\C[\param^1]$ for an affine subspace $\param^1\subset \param$, then we write $H_{\param^1}$ instead of $H_{R,c}$.

Let us recall some structural results about $H_c$. The algebra $H_c$ is filtered with $\deg \h^*=0$, 
$\deg W=0, \deg \h=1$. The associated graded is $S(\h\oplus \h^*)\#W$, \cite[Section 1.2]{EG}. This yields
the triangular decomposition $H_c=S(\h^*)\otimes \C W\otimes S(\h)$, \cite[Section 3.2]{GGOR}. Similarly,
the algebra $H_{\param}$ is filtered. We can either set $\deg \param=1$ (this is our usual choice)
or $\deg \param=0$. The latter choice shows that $H_{\param}=S(\h^*)\otimes \C[\param]W\otimes S(\h)$
as a $\C[\param]$-module.

Consider the element $\delta:=\prod_{s}\alpha_s^{\ell_s}\in S(\h^*)^W$, where $\ell_s:=\ell_{\h^s}$. Since $\operatorname{ad}\delta$ is locally nilpotent, the quotient $H_c[\delta^{-1}]$ is well-defined. There is a natural isomorphism $H_c[\delta^{-1}]\cong D(\h^{reg})\#W$, \cite[Section 1.4]{EG},\cite[Section 5.1]{GGOR}.

Consider the averaging idempotent $e:=|W|^{-1}\sum_{w\in W}w\in \C W\subset H_c$. The {\it spherical subalgebra}
by definition is $eH_ce$, it is a deformation of $S(\h\oplus \h^*)^W$. When the algebras $eH_ce$ and $H_c$
are Morita equivalent (automatically, via the bimodule $H_ce$), we say that the parameter $c$ is {\it spherical}.

There is an {\it Euler element} $h\in H_c$ satisfying $[h,x]=x, [h,y]=-y, [h,w]=0$. It is constructed as follows.
Pick a basis $y_1,\ldots,y_n\in \h$ and let $x_1,\ldots,x_n\in \h^*$ be the dual basis. For $s\in S$, let
$\lambda_s$ denote the eigenvalue of $s$ in $\h^*$ different from $1$. Then
\begin{equation}\label{eq:Euler}
h=\sum_{i=1}^n x_i y_i+\frac{n}{2}-\sum_{s\in S} \frac{2c(s)}{1-\lambda_s}s=
\sum_{i=1}^n y_ix_i-\frac{n}{2}+\sum_{s\in S}\frac{2c(s)}{1-\lambda_s^{-1}}s.\end{equation}

\subsection{Category $\mathcal{O}$}\label{SS_cat_O}
Following \cite[Section 3.2]{GGOR}, we consider the full subcategory $\OCat_c(W)$ of $H_c\operatorname{-mod}$
consisting of all modules $M$ that are finitely generated over $S(\h^*)$ and such that
$\h$ acts on $M$ locally nilpotently. Equivalently, a module $M$ lies in $\OCat_c(W)$ if it is finitely generated
over $S(\h^*)$ and is graded in such a way that the grading is compatible with that on $H_c$
induced by $\operatorname{ad}(h)$. For example, pick an irreducible representation $\lambda$ of $W$. Then the {\it Verma
module} $\Delta_c(\lambda):=H_c\otimes_{S(\h)\#W}\lambda$ (here $\h$ acts by $0$ on $\lambda$)
is in $\mathcal{O}_c(W)$. Often we drop $W$ from the notation and just write $\OCat_c$.

To a module  $M\in\OCat_c$ we can assign its associated variety $\VA(M)$ that, by definition, is the support of $M$
(as a coherent sheaf)  in $\h$. Clearly, $\VA(M)$ is a closed $W$-stable subvariety.

A basic result about $\OCat_c$ is that it is a highest weight category. Let us recall the general
definition. Let $\Cat$ be a $\C$-linear abelian category equivalent to the category of modules
over some finite dimensional associative $\C$-algebra. Let $\Lambda$ be an indexing set for the
simples in $\Cat$, we write $L(\lambda)$ for the simple object indexed by $\lambda$ and $P(\lambda)$
for its projective cover. By a highest weight category we mean a triple $(\Cat,\leqslant, \{\Delta(\lambda)\}_{\lambda\in \Lambda})$, where $\leqslant$ is a partial order on $\Lambda$ and $\Delta(\lambda), \lambda\in \Lambda,$
is a collection of {\it standard} objects in $\Cat$ satisfying the following conditions:
\begin{itemize}
\item[(i)] $\Hom_{\Cat}(\Delta(\lambda),\Delta(\mu))\neq 0$ implies $\lambda\leqslant \mu$.
\item[(ii)] $\End_{\Cat}(\Delta(\lambda))=\C$.
\item[(iii)] There is an epimorphism $P(\lambda)\twoheadrightarrow \Delta(\lambda)$ whose kernel admits
a filtration with successive quotients of the form $\Delta(\mu)$ with $\mu>\lambda$.
\end{itemize}

Axiom (iii) allows to recover projective objects from standard objects as follows. Take a linear
ordering on $\Lambda$ refining the order $\leqslant$ above: $\lambda_1>\lambda_2>\ldots>\lambda_m$.
We construct the object $P(\lambda_k)$ inductively. Set $P_k(\lambda_k):=\Delta(\lambda_k)$.
If $P_i(\lambda_k)$ with $i\leqslant k$ is already constructed, for $P_{i-1}(\lambda_k)$
we take the universal extension of $P_i(\lambda_k)$ by $\Ext^1(P_i(\lambda_k),\Delta(\lambda_{i-1}))\otimes \Delta(\lambda_{i-1})$. Then $P(\lambda_k)=P_1(\lambda_k)$.

Let us describe a highest weight structure on $\OCat_c$, \cite[Theorem 2.19]{GGOR}.
For the standard objects we take the Verma
modules. A partial order on $\Lambda=\Irr(W)$ is introduced as follows. The element
$\displaystyle \sum_{s\in S} \frac{2c(s)}{\lambda_s-1}s\in \C W$ is central so acts by a scalar,
denoted by $c_\lambda$ (and called the $c$-function), on $\lambda$. 
We set $\lambda<\mu$ if $c_\lambda-c_\mu\in \Q_{>0}$ (we could take the coarser
order by requiring the difference to lie in $\Z_{>0}$ but we do not need this). We write $<^c$ if we want to indicate the dependence on the parameter $c$. 

Since $\OCat_c$ is a highest weight category, we see that the classes $[\Delta_c(\lambda)]$ form a basis in
$K_0(\OCat_c)$. So we can identify $K_0(\OCat_c)$ with $K_0(W\operatorname{-mod})$ by sending
$[\Delta_c(\lambda)]$ to $[\lambda]$.

We will need a construction of a projective object containing  $P(\lambda)$ as a summand, \cite[Section 2.4]{GGOR}. Namely, consider
the object $\Delta_n(\lambda):= H_c\otimes_{S(\h)\#W}(\lambda\otimes S(\h)/(\h^n))$
so that $\Delta(\lambda)=\Delta_1(\lambda)$. The module $\Delta_n(\lambda)$ is graded, $\Delta_n(\lambda)=\sum_{k\in\Z} \Delta_n(\lambda)_k$,
the grading is induced from that on $H_c$ by the eigenvalues of $\operatorname{ad}(h)$.
The graded components $\Delta_n(\lambda)_k$ are finite dimensional and are preserved by the action of $h$.
Let $\tilde{\Delta}_n(\lambda)_k$ denote the generalized eigenspace for $h$
in $\Delta_n(\lambda)_k$ with eigenvalue $k+c_\lambda$. Then $\tilde{\Delta}_n(\lambda):=\bigoplus_k \tilde{\Delta}_n(\lambda)_k$ is a submodule of $\Delta_n(\lambda)$. It is not difficult to see
that a natural surjection $\tilde{\Delta}_{n+1}(\lambda)\twoheadrightarrow \tilde{\Delta}_n(\lambda)$
is an isomorphism for $n$ large enough. Denote the stable module $\tilde{\Delta}_n(\lambda)$ by
$\tilde{\Delta}(\lambda)$. It is easy to see that this module is projective and admits
a surjection onto $\Delta(\lambda)$. As a corollary of this construction we get the following.

\begin{Lem}\label{Lem:proj_summand}
There is a direct summand in $\varprojlim_{n\rightarrow \infty} H_c/H_c\h^n$ that is a projective
generator of $\OCat_c$.
\end{Lem}

Recall that in any highest weight category one has costandard objects $\nabla(\lambda), \lambda \in \Lambda,$
with $\dim \Ext^i(\Delta(\lambda),\nabla(\mu))=\delta_{i,0}\delta_{\lambda,\mu}$. In the case of the category
$\OCat_c$ one can construct the costandard objects $\nabla_c(\lambda)$ as follows. Consider the parameter
$c^*$ defined by $c^*(s):=-c(s^{-1})$. There is an isomorphism $H_c(W,\h)\xrightarrow{\sim} H_{c^*}(W,\h^*)^{opp}$
that is the identity on $\h^*,\h$  and is the inversion on $W$. Take the Verma module
$\Delta_{c^*}(\lambda^*)\in \OCat_{c^*}(W,\h^*)$,  where $\lambda^*$ is the dual representation of $\lambda$. Its restricted dual $\Delta_{c^*}(\lambda^*)^*:=
\bigoplus_k \Delta_{c^*}(\lambda^*)_k^*$ is a left $H_c(W,\h)$-module and it lies in the category $\mathcal{O}$.
This module is $\nabla_c(\lambda)$. The category $\OCat_c^{opp}$ is highest weight with the same
order and with standard objects $\nabla_c(\lambda)$.

Here are some basic properties of the standard and the costandard objects.

\begin{Lem}\label{Lem:st_basic}
The following is true:
\begin{enumerate}
\item We have $[\nabla_c(\lambda)]=[\Delta_c(\lambda)]$ for all $\lambda$.
\item If $L_c(\mu)$ lies in the socle of $\Delta_c(\lambda)$, then $\VA(L_c(\mu))=\h$.
\item If $L_c(\mu)$ lies in the head of $\nabla_c(\lambda)$, then $\VA(L_c(\mu))=\h$.
\end{enumerate}
\end{Lem}
\begin{proof}
(1) is a part of \cite[Proposition 3.3]{GGOR}. (2) follows from the observation that
$\Delta_c(\lambda)$ is torsion free over $S(\h^*)$. (3) is a corollary
of the fact that taking the restricted dual is a category equivalence $\OCat_c(W,\h)
\xrightarrow{\sim} \OCat_{c^*}(W,\h^*)$ preserving the Gelfand-Kirillov dimensions.
\end{proof}

As with an arbitrary highest weight category, we have tilting objects in $\OCat_c$. Recall that an object
is tilting if it is both standardly filtered and costandardly filtered. The indecomposable tilting objects
are again labeled by $\lambda$. More precisely, we have an indecomposable tilting $T(\lambda)$ such that
$\Delta(\lambda)\subset T(\lambda)$ and $T(\lambda)/\Delta(\lambda)$ has a filtration with successive
quotients $\Delta(\mu),\mu<\lambda$.

Here are some basic properties of the tilting objects.

\begin{Lem}\label{Lem:tilt_basic}
Let $T_c$ stand for a tilting generator -- the direct sum of all indecomposable tilting objects.
Then the  following is true:
\begin{enumerate}
\item If $\operatorname{Hom}(T_c,L_c(\mu))\neq 0$ or $\Hom(L_c(\mu),T_c)\neq 0$, then $\VA(L_c(\mu))=\h$.
\item If $\Ext^1(T_c,L_c(\mu))\neq 0$ or $\Ext^1(L_c(\mu),T_c)\neq 0$, then $\operatorname{codim}_\h \VA(L_c(\mu))\leqslant 1$.
\end{enumerate}
\end{Lem}
\begin{proof}
This is a special case of \cite[Lemma 6.2]{RSVV}.
\end{proof}

Let us now discuss the right handed analog of the category $\OCat_c$. By definition, it consists
of the finitely generated (over $S(\h^*)$) right $H_c$-modules with locally nilpotent action of $\h$. We denote this category
by $\OCat^r_c$. It also has Verma modules $\Delta^r_c(\lambda):=(\bigwedge^n\h^*\otimes\lambda^*)\otimes_{S(\h)\#W}H_c$,
where $\bigwedge^n\h\otimes\lambda$ is viewed as a left $W$-module so that $\bigwedge^n\h^*\otimes\lambda^*$ is a right $W$-module. The category $\OCat^r_c$ is highest weight with order $\leqslant^{c,r}$ given as follows. The scalar by which $\displaystyle\sum_{s\in S}\frac{2c(s)}{1-\lambda_s^{-1}}s$ acts on the right $W$-module
$\bigwedge^n\h^*\otimes\lambda^*$ coincides with $c_\lambda$. Then we set $\lambda<^{c,r}\mu$ if $c_\mu-c_\lambda\in \Q_{>0}$ (we choose the sign in this way because we are dealing with right modules so the multiplication by $x$ decreases the eigenvalue for $h$ by $1$). Note that this order is opposite to the $c$-order for $\OCat_c$. We will be mostly considering the highest weight category $\OCat^{r,opp}_c$.

To finish this section, let us note that one can also define the category $\OCat_{R,c}$ for a commutative algebra $R$
and $c:S\rightarrow R$: it consists of all $H_{R,c}$-modules that are finitely generated over $R\otimes_{\C}S(\h^*)$ and
have a locally nilpotent action of $\h$. If $R$ is a complete local Noetherian $\C$-algebra, then $\OCat_{R,c}$ is still highest weight (see, e.g., \cite[Section 4.1]{rouqqsch} for the definition of a highest weight category over a ring): the order is introduced using the $c$-function for the residue field. We will need to following special case of this construction. Let $\ell\subset \param$ be a Weil generic line passing through $c\in \param$ (the meaning of ``Weil generic'' is explained in Section \ref{SS_ideas}). Let $\hbar$ be an affine coordinate on $\ell$
vanishing on $c$. We take $R=\C[[\hbar]](=\C[\ell]^{\wedge_c})$ and form the Cherednik algebra 
for the natural map $S\rightarrow R$ to be denoted by $\tilde{H}_c$. The corresponding category
$\mathcal{O} $ will be denoted by $\tilde{\OCat}_{c}$ (or $\tilde{\OCat}_{c}(W)$). 

\subsection{KZ functor}\label{SS_KZ}
Recall that $H_c[\delta^{-1}]\cong D(\h^{reg})\# W$. So we have the localization functor $\loc:
\OCat_c\rightarrow \LS^W(\h^{reg})$, where on the right hand we have the category of $W$-equivariant
local systems on $\h^{reg}$. The functor is given by $\loc(M):=M[\delta^{-1}]$. There is a standard
equivalence $\LS^W(\h^{reg})\xrightarrow{\sim} \LS(\h^{reg}/W), N\mapsto eN$. One can show that the image
of $\OCat_c$ lies in the subcategory $\LS_{rs}(\h^{reg}/W)$ of the local systems with regular singularities,
see \cite[Proposition 5.7]{GGOR}. The latter is equivalent to the category $\C B_W\operatorname{-mod}$ of the finite dimensional $B_W$-modules, the equivalence sends a local section to its fiber at a point equipped with the monodromy action.

It follows from \cite{Hecke_dim} that the essential image of the functor $\OCat_c\rightarrow
\C B_W\operatorname{-mod}$ coincides with $\underline{\mathcal{H}}_q\operatorname{-mod}$. The parameter $q$
is computed as follows. We can find elements $h_{H,j}\in \C$  with $j=0,\ldots,\ell_{H}-1$
and $h_{H,j}=h_{H',j}$ for $H'\in WH$  such that
\begin{equation}\label{eq:c_to_h}c(s)=\sum_{j=1}^{\ell-1}\frac{1-\lambda_s^j}{2}(h_{\h^s,j}-h_{\h^s,j-1})\end{equation}
Clearly, for fixed $H$, the numbers $h_{H,0},\ldots, h_{H,\ell_H-1}$ are defined up to a common summand.
We can recover the elements $h_{H,i}$ by the formula
\begin{equation}\label{eq:h_to_c} h_{H,i}=\frac{1}{\ell_H}\sum_{s\in W_H\setminus \{1\}}\frac{2c(s)}{\lambda_s-1}\lambda_s^{-i}
\end{equation}
Note that $\sum_{i=0}^{\ell_H-1}h_{H,i}=0$. We will view $h_{H,i}$ as an element of $\param^*$
whose value on $c:S\rightarrow \C$ is given by (\ref{eq:h_to_c}).

We set
\begin{equation}\label{eq:h_to_q}q_{H,j}:=\exp(2\pi\sqrt{-1}(h_{H,j}+j/\ell_H)).\end{equation}
So we get the functor $\KZ:\OCat_c\rightarrow \underline{\mathcal{H}}_q\operatorname{-mod}$. Let us list
properties of this functor obtained in \cite[Section 5]{GGOR}.

\begin{Prop}\label{Prop:KZ_prop}
The functor $\KZ$ has the following properties:
\begin{enumerate}
\item $\KZ$ is a quotient functor, its kernel consists of all modules  $M\in \OCat_c$ that are torsion over
$S(\h^*)$ ($\Leftrightarrow \VA(M)\neq \h$).
\item $\KZ$ is fully faithful on the projective objects. Also it is fully faithful on the tilting objects.
\item Suppose that we have $q_{H,i}\neq q_{H,j}$ for any reflection hyperplane $H$ and $i\neq j$.
Then $\KZ$ is fully faithful on the standardly filtered objects (=the objects admitting a filtration
with standard successive quotients).
\end{enumerate}
\end{Prop}

Note that (2) and (3)  were established in \cite{GGOR} for the functor $\OCat_c\rightarrow \mathcal{H}_q\operatorname{-mod}$. But since $\underline{\mathcal{H}}_q\operatorname{-mod}$
is a full (actually, Serre) subcategory of $\mathcal{H}_q\operatorname{-mod}$, (2) and (3) are also true
for our functor.

Let $P_{KZ}$ denote the projective object in $\OCat_c(W)$ defining the functor $\KZ$ so that
there is a distinguished isomorphism $\underline{\mathcal{H}}_q(W)\xrightarrow{\sim}\End(P_{KZ})^{opp}$.
The object $P_{KZ}$ is the sum of all objects in $\OCat_c$ that are simultaneously
projective and injective (hence tilting) with suitable multiplicities.

We also have a version of $\KZ$ over rings. Namely, let $R$ be a regular complete local ring with residue field
$\C$. Then the exponential map still makes sense and we get a quotient functor $\KZ: \OCat_{R,c}(W)
\twoheadrightarrow \underline{\mathcal{H}}_{R,q}(W)\operatorname{-mod}$. The properties
(1)-(3) still hold for this functor $\KZ$.

The Hecke algebra for $R$ as in the very end of Section \ref{SS_cat_O} will be denoted 
by $\tilde{\mathcal{H}}_q(W)$.  

\subsection{Ringel duality}\label{SS_Ringel}
Let $\Cat_1,\Cat_2$ be two highest weight categories. Let $\Cat_2^\Delta, \Cat_1^\nabla$ denote the full
subcategories of standardly and costandardly filtered objects in $\Cat_2,\Cat_1$, respectively. Let
$R$ be an equivalence  $\Cat_1^\nabla\xrightarrow{\sim} \Cat_2^\Delta$ of exact categories.
Let $T$ denote the tilting generator
of $\Cat_1$, i.e., the sum of all indecomposable tilting objects. Then $\Cat_2$ gets identified
with $\End(T)^{opp}\operatorname{-mod}$ and the equivalence $R$ above becomes $\Hom(T,\bullet)$.
We also have a derived equivalence $R\Hom(T,\bullet): D^b(\Cat_1)\rightarrow D^b(\Cat_2)$. This equivalence
maps injectives to tiltings and, obviously, tiltings to projectives.  We write $\Cat_1^\vee$ for $\Cat_2$.
The functor $R$ is called the (covariant) {\it Ringel duality}, and the category $\Cat_1^\vee$
is called the {\it Ringel dual} of $\Cat_1$.

In the case when $\Cat_1=\OCat_c$ the Ringel duality was realized explicitly in \cite[Section 4.1]{GGOR}.
Namely, set $n:=\dim\h$ and consider the functor $D:=R\Hom_{H_c}(\bullet, H_c)[n]$.
It defines a derived equivalence between $D^b(\OCat_c)$ and $D^b(\OCat_c^{r,opp})$,
that maps $\Delta_c(\lambda)$ to $\Delta_c^r(\lambda)$. Hence the functor $D$ realizes
$R^{-1}$ and $\OCat_c=\OCat_c^{r,opp,\vee}$.

An important property of the functor $D$ is that it is a perverse equivalence with respect to 
filtrations by the dimension of support. Namely, set $\Cat^1:=\OCat_{c}, \Cat^2:=\OCat_c^{r,opp}$. Consider a filtration $\Cat^1:=\Cat^1_0\supset\Cat^1_1\supset \Cat^1_2\supset\ldots\supset\Cat^1_j\supset\ldots\supset \Cat^1_{n+1}:=\{0\}$,
where, by definition, $\Cat^1_j$ consists of all modules $M\in \Cat^1$ with $\dim \VA(M)\leqslant n-j$. Define the filtration $\OCat_{c}^{r,opp}=\Cat^2_0\supset \Cat^2_1\supset\ldots \supset\Cat^2_{n+1}=\{0\}$ similarly.
The definition of a perverse equivalence  was given in Section \ref{SS_perv_intro}.

\begin{Lem}\label{Lem:D_perv}
The equivalence $D: D^b(\OCat_c)\rightarrow D^b(\OCat_{c}^{r,opp})$ is perverse with respect to the filtrations 
introduced above.
\end{Lem}
\begin{proof}
Pick $M\in \OCat_c$ and equip it with a good filtration. Let $H_{\hbar,c}$ denote the Rees algebra,
and $M_\hbar$ be the Rees module over $H_{\hbar,c}$ constructed from the filtration on $M$.
The right $H_{\hbar,c}$-module $\Ext^i_{H_{\hbar,c}}(M_\hbar,H_{\hbar,c})$ is graded and
$\Ext^i_{H_{\hbar,c}}(M_\hbar,H_{\hbar,c})/(\hbar-1)=\Ext^i_{H_c}(M,H_c)$. So $\Ext^i_{H_c}(M,H_c)$
is equipped with a filtration. The module $\Ext^i_{H_{\hbar,c}}(M_\hbar,H_{\hbar,c})$ is finitely generated
so the filtration is good. Further, we have a standard short exact sequence
$$\Ext^i_{H_{\hbar,c}}(M_\hbar,H_{\hbar,c})\xrightarrow{\hbar} \Ext^i_{H_{\hbar,c}}(M_\hbar,H_{\hbar,c})
\rightarrow \Ext^i_{H_{\hbar,c}}(M_\hbar,H_{\hbar,c}/(\hbar))$$
Note that the last term is naturally identified with $\Ext^i_{S(\h\oplus \h^*)\#W}(\gr M, S(\h\oplus \h^*)\#W)$.
We conclude that $\gr\Ext_{H_c}^i(M,H_c)\subset \Ext^i_{\gr H_c}(\gr M,\gr H_c)$.

By standard Commutative algebra results, if the support of $\gr M$ has codimension $n+s$,
then $\Ext^k_{S(\h\oplus \h^*)}(\gr M,S(\h\oplus \h^*))=0$ for $k<n+s$, and $\Ext^k_{S(\h\oplus \h^*)}(\gr M,S(\h\oplus \h^*))$ has support of codimension larger than $n+s$ for $k>n+s$. Since $\gr H_c=S(\h\oplus \h^*)\#W$, we have
that $\Ext^k_{\gr H_c}(\gr M,\gr H_c)= \Ext^k_{S(\h\oplus \h^*)}(\gr M, S(\h\oplus \h^*)^{\oplus |W|})^W$.
So (I) and (II) in the definition of a perverse equivalence hold for $D$. (III) follows from $D^2=\operatorname{id}$
and a standard spectral sequence for the composition of derived functors.
\end{proof}

We see that the inverse Ringel duality $R^{-1}$ is perverse.

The homological duality extends to the deformations $\tilde{\OCat}_{c}, \tilde{\OCat}_{c}^{r,opp}$, this is
the derived equivalence $D:=R\Hom_{\tilde{H}_{c}}(\bullet, \tilde{H}_{c}):D^b(\tilde{\OCat}_{c})\xrightarrow{\sim}
D^b(\tilde{\OCat}_{c}^{r,opp})$. This equivalence is again perverse with respect to the filtrations
by dimension of support, this is proved exactly as in the undeformed setting.

\subsection{Remarks on  orderings and parameterizations}\label{SS_param}
We consider the $\Z$-lattice and the $\Q$-lattice $\param^*_\Z\subset \param^*_{\Q}\subset \param^*$
spanned by the elements $h_{H,i}-h_{H,j}$ and the dual lattices $\param_{\Z}\subset \param_{\Q}\subset \param$.
So $c-c'\in \param_{\Z}$ if and only if there are nonzero
scalars $\gamma_H, H\in \mathfrak{H}/W,$ such that $q'_{H,i}=\gamma_H q_{H,i}$ for all $H$
and $i$ (and so the algebras $\mathcal{H}_q$ and $\mathcal{H}_{q'}$ are identified).

We will need a certain sublattice  in  $\param_{\Z}$.  In \cite[Section 7.2]{BC}, Berest and Chalykh 
established a group homomorphism $\mathsf{tw}:\param_\Z\rightarrow \operatorname{Bij}(\Irr W)$ 
called the {\it KZ twist}. Set $\underline{\param}_{\Z}:=\ker \mathsf{tw}$.

We will use another spanning set for $\param_{\Z}$. We can assign an element in $\param_{\Z}$ to a one-dimensional character of $W$ as follows. There is a homomorphism  $\operatorname{Hom}(W,\C^\times)
\rightarrow \prod_{H\in \mathfrak{H}/W} \operatorname{Irr}(W_H)$  given by
the restriction. It turns out that this map is an isomorphism, see \cite[3.3.1]{rouqqsch}. So to an arbitrary
collection of elements $(a_H)$ with $0\leqslant a_H\leqslant \ell_H-1$ we can assign the character of $W$
that sends $s$ to $\lambda_s^{-a_H}$. To a character $\chi$ given in this form we assign the element
$\bar{\chi}\in \param$ given by $h_{H,i}(\bar{\chi})=1-\frac{a_H}{\ell_H}$ if $i\geqslant \ell-a_H$
and $-\frac{a_H}{\ell_H}$ if $i<\ell-a_H$. The motivation behind this definition will
be explained in Section \ref{SS_HC_basic}. Clearly, the  elements of the form $\bar{\chi}$
span $\param_{\Z}$.

Let us proceed to orders.

\begin{Lem}\label{Lem:c_fun_int}
The function $c\mapsto c_\lambda$ is rational on $\param_{\Q}$.
\end{Lem}
\begin{proof}
The action of the element $\varphi_{H}=\sum_{s\in W_H\setminus \{1\}} \frac{2c(s)}{\lambda_s-1}s$
on the $W_H$-isotypic component corresponding to the character $s\mapsto \lambda_s^j$ is by the scalar
$\ell_H h_{H,-j}$. The claim follows.
\end{proof}

Define an equivalence relation $\sim$ on $\Irr(W)$ by setting
$\lambda\sim \lambda'$ if $c_{\lambda}=c_{\lambda'}$ for every parameter $c$. Note that different one-dimensional
representations cannot be equivalent. Now if $\lambda\not\sim \mu$,
then we have the hyperplane $\Pi_{\lambda,\mu}$ in $\param$ given by $c_{\lambda}=c_{\mu}$.
All the hyperplanes $\Pi_{\lambda,\mu}$  are rational.

Fix a coset $c+\param_{\Z}$ and consider $c'$ in this coset.
We write $c\prec c'$ if $\lambda\leqslant^c \lambda'$ implies $\lambda\leqslant^{c'}\lambda'$.
We write $c\sim c'$ if $c\prec c'$
and $c'\prec c$. The equivalence classes are relative interiors in the cones defined by the hyperplane
arrangement $\{\Pi_{\lambda,\mu}, \lambda\not\sim \mu, c_\lambda-c_\mu\in \Q\}$
on $c+\param_{\Z}$. We are mostly interested in the
open cones. Below the open cones in this stratification will be called {\it open chambers}.
For each open chamber we have its opposite chamber, where the order is opposite.
Note that if $c$ is Weil generic in $\param$,
we have just one open chamber because there are no non-equivalent pairs $\lambda,\mu\in \Irr(W)$
with $c_\lambda-c_\mu\in \Q$ (the locus where such a pair exists is the countable union of
hyperplanes, and we just take $c$ outside of this union). Similarly, for a Weil generic $c$ on a rational hyperplane parallel to $\Pi_{\lambda,\mu}$ we have exactly two  open cones that are opposite to each other.

\subsection{Rouquier equivalence theorem}
In \cite[Section 4.2]{rouqqsch}, Rouquier established some tools to prove an equivalence of categories $\OCat_{c},
\OCat_{c'}$ for different parameters $c,c'$.

We start with a general setting. Let $A_\hbar$ be a $\C[[\hbar]]$-algebra that is  free of finite rank as a module
over $\C[[\hbar]]$. Assume that $A_{\hbar}[\hbar^{-1}]$ is split semisimple. Let $B_\hbar$ be another
$\C[[\hbar]]$-algebra (free of finite rank) and let $P_\hbar$ be a projective $B_\hbar$-module with a fixed
isomorphism
$\End_{B_\hbar}(P_\hbar)^{opp}\cong A_\hbar$. Assume that $P_\hbar[\hbar^{-1}]$ is a projective generator
of $B_\hbar[\hbar^{-1}]\operatorname{-mod}$. So we have an exact functor $\pi_\hbar=\Hom_{B_\hbar}(P_\hbar,\bullet): B_\hbar\operatorname{-mod}
\twoheadrightarrow A_\hbar\operatorname{-mod}$ that is an equivalence after inverting $\hbar$. Next, suppose that
$B_\hbar\operatorname{-mod}$ is a highest weight category over $\C[[\hbar]]$ with $\Lambda$ being an indexing
set of simples.

Let $B,A,\pi$ be the specializations of $B_\hbar,
A_\hbar,\pi_\hbar$ to $\hbar=0$. Note that functor $\pi_\hbar$ defines a bijection between $\Lambda$ and the set of
simple $A_\hbar[\hbar^{-1}]$-modules  given by $\lambda\mapsto \Delta_\hbar(\lambda)[\hbar^{-1}]$.

We say that $\pi$ is $0$-faithful if it is fully faithful on standardly
filtered objects.

The following result is due to Rouquier, \cite[Proposition 4.42, Theorem 4.49]{rouqqsch}.

\begin{Prop}\label{Prop:abstr_equiv}
Let $(B_\hbar, P_\hbar, \pi_\hbar)$ and $(B'_\hbar, P'_\hbar, \pi'_\hbar)$ be two triples as above. Suppose that
the following hold:
\begin{enumerate}
\item The functors $\pi,\pi'$ are $0$-faithful.
\item There is an order on $\operatorname{Irr}(B)\cong \operatorname{Irr}(A_\hbar[\hbar^{-1}])\cong
\operatorname{Irr}(B')$ that is highest weight for both $B\operatorname{-mod}$ and $B'\operatorname{-mod}$.
\end{enumerate}
Then there is an equivalence $B_\hbar\operatorname{-mod}\xrightarrow{\sim} B'_{\hbar}\operatorname{-mod}$
that intertwines the functors $\pi_\hbar,\pi'_\hbar$.
\end{Prop}

The proof goes as follows. First, one notices that $\pi_\hbar(\Delta_\hbar(\lambda))\cong \pi'_\hbar(\Delta'_\hbar(\lambda))$, this follows
from \cite[Lemma 4.48]{rouqqsch} and  uses only (2). Then one shows that the functors $\pi_\hbar, \pi'_\hbar$ are 1-faithful, i.e.,
preserve both Hom's and $\Ext^1$'s between standardly filtered objects, \cite[Proposition 4.42]{rouqqsch},
this follows from (1).  Finally,  one uses a construction of the indecomposable projectives
in $B_\hbar\operatorname{-mod}$ recalled in Section \ref{SS_cat_O}, and gets $\pi_\hbar(P_\hbar(\lambda))\cong
\pi'_\hbar(P'_\hbar(\lambda))$. This completes the proof.

Rouquier applied this result to  Cherednik categories $\OCat$, \cite[Theorem 5.5]{rouqqsch} (condition
(ii) was missing in {\it loc.cit.}).

\begin{Prop}\label{Prop:Chered_equi}
Suppose that $c,c'\in \param$ satisfy the following conditions:
\begin{itemize}
\item[(i)] $c-c'\in \param_\Z$.
\item[(ii)] $\mathsf{tw}(c'-c)=\operatorname{id}$.
\item[(iii)] The ordering $\leqslant^c$ refines $\leqslant^{c'}$.
\item[(iv)] For each hyperplane $H$ and different $i,j$, we have $q_{H,i}\neq q_{H,j}$.
\end{itemize}
Then there is an equivalence $\OCat_c\xrightarrow{\sim} \OCat_{c'}$ of highest weight
categories mapping $\Delta_c(\lambda)$ to $\Delta_{c'}(\lambda)$ that intertwines
the KZ functors $\OCat_c,\OCat_{c'}\twoheadrightarrow \underline{\mathcal{H}}_q\operatorname{-mod}$.
\end{Prop}

Below we will see that  assumption (iv) is not necessary.

\subsection{Induction and restriction functors for category $\mathcal{O}$}\label{SS_ind_res_O}
Let $W'\subset W$ be a parabolic subgroup.
We have a natural homomorphism $\underline{\mathcal{H}}_q(W')\rightarrow \underline{\mathcal{H}}_q(W)$
that gives rise to the restriction functor ${}^{\mathcal{H}}\Res_{W}^{W'}:\underline{\mathcal{H}}_q(W)
\operatorname{-mod}\rightarrow \underline{\mathcal{H}}_q(W')\operatorname{-mod}$.
When we write $q$ in $\mathcal{H}_q(W')$ we mean
the parameter $q'$ given by $q'_{H,i}=q_{H,i}$ for every reflection hyperplane of $W'$.
The functor ${}^{\mathcal{H}}\Res_{W}^{W'}$ has  left (induction) ${}^{\mathcal{H}}\Ind_{W'}^{W}$
and right (coinduction) ${}^{\mathcal{H}}\operatorname{Coind}_{W'}^{W}$ adjoint functors.

On the other hand, in \cite[Section 3.5]{BE}, Bezrukavnikov and Etingof defined induction and restriction
functors for the Cherednik categories $\mathcal{O}$. Namely, we have the restriction
functor ${}^{\OCat}\Res_{W}^{W'}:\OCat_c(W)\rightarrow \OCat_{c}(W')$ and its left adjoint,
the induction functor ${}^{\OCat}\Ind_{W'}^W:\OCat_{c}(W')\rightarrow \OCat_{c}(W)$.

Let us recall the construction of ${}^{\OCat}\operatorname{Res}_W^{W'}$ from \cite{BE}.
Pick $b\in \h$ such that $W_b=W'$. We can consider the completion $H_c^{\wedge_b}:=\C[\h/W]^{\wedge_b}\otimes_{\C[\h/W]}H_c$, where $\C[\h/W]^{\wedge_b}$ is
the completion of $\C[\h/W]$ with respect to the maximal ideal defined by $b$.
The completion $H_c^{\wedge_b}$ is a filtered algebra. Similarly, we can consider the
completion $H_c(W',\h)^{\wedge_0}:=\C[\h/W']^{\wedge_0}\otimes_{\C[\h/W']}H_c(W',\h)$
(where we write $H_c(W',\h)$ for $H_c(W')\otimes D(\h^{W'})$)
and form the centralizer algebra $Z(W,W',H_c(W',\h)^{\wedge_0})$ as in \cite[Section 3.2]{BE},
as an algebra, this is just $\operatorname{Mat}_{|W/W'|}(H_c(W',\h)^{\wedge_0})$.
Bezrukavnikov and Etingof, \cite[Section 3.3]{BE}, produced an explicit filtration preserving isomorphism $\theta_b: H_c(W,\h)^{\wedge_b}\xrightarrow{\sim} Z(W,W', H_c(W',\h)^{\wedge_0})$. Form the completion $M^{\wedge_b}:=
\C[\h/W]^{\wedge_b}\otimes_{\C[\h/W]}M$. This gives rise to a functor from $\OCat_c$ to the category
$\OCat(H_c^{\wedge_b})$ of all $H_c^{\wedge_b}$-modules finitely generated over $S(\h^*)^{\wedge_b}$.
Then we take elements in $M':=e(W')\theta_{b*}(M^{\wedge_b})$
that are finite for the action of the Euler element of $H_c(W',\h)$ (here $e(W')$ is a primitive
idempotent in $Z(W,W', \C W')$ defining a Morita equivalence between $Z(W,W',\C W')$ and $\C W'$). Let $M'_{fin}$ be the resulting
$H_c(W',\h)$-module. Then $M'_{fin}=\C[\h^{W'}]\otimes {}^{\OCat}\Res_{W}^{W'}(M)$. Note that
the functor ${}^{\OCat}\Res_W^{W'}$ is the composition of $M\mapsto M^{\wedge_b}$
and a category equivalence $\OCat(H_c^{\wedge_b})\cong \OCat_c(W')$.

It was shown in \cite{fun_iso}, see also \cite[Section 2.4]{Shan} for the same result but under additional
restrictions, that the functors ${}^{\OCat}\Res_{W}^{W'}$ and ${}^{\OCat}\Ind_W^{W'}$ are biadjoint.
Moreover, Shan has checked in \cite[Theorem 2.1]{Shan} that the restriction functors intertwine the KZ functors:
$\KZ'\circ {}^\OCat\Res_{W}^{W'}={}^{\mathcal{H}}\Res_{W}^{W'}\circ \KZ$.
Here $\KZ'$ stands for the KZ functor $\OCat_c(W')\rightarrow \mathcal{H}_q(W')\operatorname{-mod}$.

It is clear from the construction in \cite[Section 3.5]{BE} that the functor ${}^{\OCat}\Ind_{W'}^W$ maps the category of the torsion $S(\h_{W'}^*)^{W'}$-modules to the category of the torsion $S(\h^*)^W$-modules. So it descends
to the functor $\underline{\mathcal{H}}_q(W')\operatorname{-mod}\rightarrow \underline{\mathcal{H}}_q(W)\operatorname{-mod}$. It follows that ${}^{\mathcal{H}}\Res_{W}^{W'}$ admits a biadjoint functor, so ${}^{\mathcal{H}}\Ind_{W'}^{W}\cong {}^{\mathcal{H}}\operatorname{Coind}_{W'}^{W}$.
We also note that the  induction functors  intertwine the KZ functors.

Let us note that the induction and restriction functors extends to the deformed categories:
we have biadjoint functors ${}^\OCat\Res_{W}^{W'}:\tilde{\OCat}_{c}(W)\rightarrow \tilde{\OCat}_{c}(W')$
and ${}^\OCat\Ind_{W'}^{W}:\tilde{\OCat}_{c}(W')\rightarrow \tilde{\OCat}_{c}(W)$. They intertwine
the KZ functors.

For $M\in \OCat_c(W)$, the associated variety $\VA(M)$ is the union of the strata of the form
$W\h^{W'}$, where $W'$ is a parabolic subgroup. A stratum $W\h^{W'}$ is the union of irreducible
components of $\VA(M)$ if and only if ${}^{\OCat}\Res_{W}^{W'}(M)$ is finite dimensional and nonzero.

\section{Harish-Chandra bimodules}\label{S_prelim_HC}
\subsection{Harish-Chandra bimodules}\label{SS_HC_basic}
In this section, we recall the definition and some basic results about Harish-Chandra (HC) bimodules over the algebras $H_{R,c}$ and $eH_{R,c}e$ (we write $H_c$ for $H_c(W)$, etc.).

A definition of a HC $H_{c'}$-$H_{c}$-bimodule was introduced in \cite[Section 3]{BEG}. A HC bimodule, by definition,  is a finitely generated $H_{c'}$-$H_c$-bimodule, where the adjoint actions of $S(\h^*)^W, S(\h)^W$ are locally nilpotent.
Note that, by the definition, the Harish-Chandra bimodules form a Serre subcategory inside the category of
all $H_{c'}$-$H_c$-bimodules.

Here are some basic properties of HC bimodules.

\begin{Prop}\label{Prop_HC_prop} Let $M$ be a HC $H_{c'}$-$H_c$-bimodule. Then the following is true:
\begin{enumerate}
\item $M$ is finitely generated as a left $H_{c'}$-module, as a right $H_{c}$-module and as a $S(\h^*)^W\otimes S(\h)^W$-module (with $S(\h^*)^W\subset H_{c'}, S(\h)^W\subset H_c$).
\item If $N$ is a HC $H_{c''}$-$H_{c'}$-bimodule, then $N\otimes_{H_{c'}}M$ is also HC.
\item If $N\in \OCat_c$, then $M\otimes_{H_c}N\in \OCat_{c'}$.
\end{enumerate}
\end{Prop}
\begin{proof}
(1) is a part of  \cite[Lemma 3.3]{BEG}. (2) is straightforward. In (3) notice that $M\otimes_{H_c}N$ is finitely
generated over $S(\h^*)^W$ thanks to (1) and has locally nilpotent action of the augmentation ideal $S(\h)^W_+\subset S(\h)^W$. The latter easily implies that the action of $\h$ is also locally nilpotent.
\end{proof}

We can give an analogous definition for $H_{\param}$-bimodules. Namely, we say that an $H_{\param}$-bimodule
$M$ is HC if there is $\psi\in \param$ such that $pm=m(p-\langle\psi,p\rangle)$ for any $p\in \param^*$
and the adjoint actions of $S(\h^*)^W, S(\h)^W$ are locally nilpotent. Let $\operatorname{HC}(H_{\param},\psi)$ denote the category of such HC bimodules. We could  relax the condition on the compatibility between the left and the  right $\C[\param]$-actions but this is technical. Also we can speak about HC bimodules over the spherical subalgebras.

Let us provide an important example. Let $\chi$ be a character of $W$, $e_\chi\in \C W$ be the corresponding idempotent, $\bar{\chi}$ be the  element in $\param_\Z$ constructed in Section \ref{SS_param}.
According to \cite[Section 5.4]{BC}, there is an isomorphism
$\varphi: e H_{\param}e\xrightarrow{\sim} e_\chi H_{\param} e_{\chi}$ that maps $p\in \param^*$ to
$p+\langle\bar{\chi},p\rangle$.

\begin{Lem}\label{Lem:shift_HC} $e H_{c+\bar{\chi}} e_{\chi}$  is a HC $eH_{c+\bar{\chi}}e$-$eH_ce$-bimodule.
\end{Lem}
\begin{proof}
According to the construction of the isomorphism $\varphi$ in \cite[Section 5.4]{BC}, this isomorphism preserves
the filtrations given by $\deg \h,\deg W=0, \deg \h^*=1$ and the gradings induced by $\operatorname{ad}h$.
Moreover, the associated graded isomorphism $S(\h\oplus \h^*)^W=\gr eH_ce\xrightarrow{\sim}\gr e_\chi H_{c+\bar{\chi}} e_\chi=S(\h\oplus \h^*)^W$ is the identity. The associated graded of $e H_{c+\bar{\chi}}e_\chi$ is  the $S(\h\oplus\h^*)^W$-bimodule $S(\h\oplus \h^*)^{W,\chi^{-1}}$. Pick a homogeneous element $a\in S(\h)^W$.
The operator induced by $[a,\cdot]$ on $\operatorname{gr}e H_{c+\bar{\chi}}e_\chi=
S(\h\oplus \h^*)^{W,\chi^{-1}}$ is zero hence $[a,\cdot]$ is locally nilpotent. Now let us pick a homogeneous
element $b\in S(\h^*)^W$ and prove that $[b,\cdot]$ is locally nilpotent on $e H_{c+\bar{\chi}}e_\chi$.
The bimodule $e H_{c+\bar{\chi}}e_\chi$ is graded and the grading is compatible with the filtration.
Then we can twist the filtration using the grading, compare to Remark \ref{Rem:diff_filt}
below, so that the multiplication by $\h^*$ preserves the filtration, and this does not change the
associated graded. This shows that $[b,\cdot]$  is locally nilpotent.
\end{proof}

So we get an $H_{c+\bar{\chi}}$-$H_c$ bimodule $$\mathcal{B}_{c,\bar{\chi}}:=
H_{c+\bar{\chi}}e\otimes_{eH_{c+\bar{\chi}}e}e H_{c+\bar{\chi}} e_{\chi}\otimes_{eH_ce}eH_c.$$
Similarly, we get the $H_{c}$-$H_{c+\bar{\chi}}$ bimodule $\mathcal{B}_{c+\bar{\chi},-\bar{\chi}}$.
These bimodules are HC. We also can define the objects $\B_{\param,\bar{\chi}}\in \HC(H_{\param},\bar{\chi}),
\B_{\param,-\bar{\chi}}\in \HC(H_{\param}, -\bar{\chi})$. That these bimodules are HC follows
from Lemma \ref{Lem:shift_HC}: if the adjoint actions of $S(\h)^W, S(\h^*)^W$ are locally nilpotent
on all fibers over $\param$, then they are locally nilpotent on the whole bimodule.

There is an alternative definition of HC bimodules given in \cite[Section 3.4]{sraco}. Equip the algebra $H_{\param}$
with a filtration, $H_\param=\bigcup_{i\geqslant 0}\operatorname{F}_{i}H_\param$, by setting $\deg \h=\deg\param=1, \deg \C W=\deg \h^*=0$.
The algebra $\gr H_{\param}$ is finite over its center
denoted
by $Z_{\param}$ (recall the Satake isomorphism from \cite[Theorem 3.1]{EG},
$Z_{\param}\cong e(\gr H_{\param})e$,
given by $z\mapsto ze$). By a Harish-Chandra $H_{\param}$-bimodule
we mean a bimodule $M$ that can be equipped with an increasing filtration $M=\bigcup_{i\geqslant m} M_{\leqslant i}$
such that $\gr M$ is finitely generated over $\gr H_{\param}$ and, moreover, the left and the right
actions of $Z_{\param}$ coincide. Such a filtration is called {\it good}.
One can give a definition of a HC $eH_{\param}e$-bimodule in a similar fashion.

\begin{Rem}\label{Rem:diff_filt}
Let us remark that we used a different filtration, denote it here by $\operatorname{F}'$, in \cite[Section 3.4]{sraco}.
The filtrations are related as follows:
$$\operatorname{F}'_i H_\param=\bigoplus_{k}(\operatorname{F}_k H_\param)\cap \{a\in H_\param| [h,a]=(i-2k)a\}.$$
It follows that our present definition is equivalent to a more technical definition from \cite[Section 3.4]{sraco}.
\end{Rem}

We have checked in \cite[Proposition 5.4.3]{sraco} that any HC bimodule in the sense of \cite{BEG} is also HC in the sense
of \cite{sraco}. Conversely, let $M$ be a HC bimodule in the sense of \cite{sraco} such that
$pm=m(p-\langle \psi,p\rangle)$. Then $M\in \HC(H_{\param},\psi)$, see
\cite[Proposition 5.4.1]{sraco}.

To $M\in \HC(H_{\param},\psi)$ we can assign its associated variety, $\VA(M)\subset (\h\oplus \h^*)/W$. By definition, this
is the support of $\gr M/ \param \gr M$, where the associated graded is taken with respect to
a good filtration.

There is one important property of HC bimodules that is easy to see from the definition in \cite{sraco}.
Namely, for a HC $H_{\param}$-bimodule $M$ we can consider its specialization $M_c:= M\otimes_{\C[\param]}\C_c$.
By the right support of $M$ we mean $\operatorname{Supp}^r(M):=\{c\in \param| M_c\neq 0\}$.
The following lemma is proved completely analogously to \cite[Lemma 5.7, Corollary 5.8]{BL}.

\begin{Lem}\label{Lem:HC_supp}
Let $\param^1\subset\param$ be an affine subspace, $\psi\in \param$,  and let $M\in \HC(H_{\param^1},\psi)$. Then  the following is true:
\begin{enumerate}
\item There is $f\in \C[\param^1]$ such that $M\otimes_{\C[\param^1]}\C[\param^1_f]$ is a
free module over $\C[\param^1_f]$. Here $\param^1_f$ is the principal open subset in
$\param^1$ defined by $f$.
\item The support $\operatorname{Supp}^r(M)$ is a constructible subset of $\param^1$.
\end{enumerate}
\end{Lem}

Let us deduce some corollaries from Lemma \ref{Lem:HC_supp}.

\begin{Cor}\label{Cor:spherical}
The following is true:
\begin{enumerate}
\item There is a Zariski open subset of $\param$ consisting of spherical parameters.
\item Let $\chi$ be a character of $W$. Then there is a Zariski open subset of parameters
$c$ such that $\mathcal{B}_{c,\bar{\chi}}$ and $\mathcal{B}_{c+\bar{\chi},-\bar{\chi}}$
are mutually inverse Morita equivalences.
\end{enumerate}
\end{Cor}
\begin{proof}
The algebra $H_c$ is simple for a Weil generic $c$, see, e.g., \cite[Section 4.2]{sraco}, so such $c$ is
spherical. (1) follows from (2) of Lemma \ref{Lem:HC_supp} applied to the bimodule
$H_{\param}/H_{\param}eH_{\param}$.

Let us proceed to (2). Note that we have  natural homomorphisms
$$\B_{\param,-\bar{\chi}}\otimes_{H_\param}\B_{\param,\bar{\chi}}\rightarrow H_{\param}, \B_{\param,\bar{\chi}}\otimes_{H_{\param}}\B_{\param,-\bar{\chi}}\rightarrow H_{\param}.$$
Their specializations to Weil generic $c$ are isomorphisms because  they are always isomorphisms
after inverting $\delta$ and $H_c$ is simple. So they are also isomorphisms for a  Zariski
generic $c$.
\end{proof}

%
%
\begin{Rem}\label{Rem:supp_O}
It is easy to see that a direct analog of Lemma \ref{Lem:HC_supp} holds for the category $\mathcal{O}_{\param^1}$.
\end{Rem}

\subsection{Restriction functors for HC bimodules: construction}\label{SS_dag_constr}
Pick a parabolic subgroup  $W'\subset W$. Set $\Xi:=N_W(W')/W'$. Let $\h_{W'}$ denote the unique $W'$-stable complement to $\h^{W'}$, the spaces $\h_{W'},\h^{W'}$ are $N_W(W')$-stable. Form the algebra $H_{\param}(W')$ for
the $W'$-action on $\h_{W'}$. By definition $H_{\param}(W'):=\C[\param]\otimes_{\C[\param^1]}H_{\param^1}(W')$,
where $\param^1$ is the parameter space for $W'$, it comes with a natural linear map $\param\rightarrow \param^1$
induced by the restriction from $S$ to $S\cap W'$. We consider the category of $\Xi$-equivariant $H_{\param}(W')$-modules,
by definition, it consists of the HC $H_{\param}(W')$-bimodules $N$ equipped with a $N_W(W')$-action that
\begin{itemize}
\item restricts to the adjoint $W'$-action,
\item and makes the structure map $H_{\param}(W')\otimes N\otimes H_{\param}(W')\rightarrow N$
equivariant for the $N_W(W')$-action.
\end{itemize}
We denote the category of $\Xi$-equivariant HC $H_{\param}(W')$-bimodules by $\HC^\Xi(H_{\param}(W'))$.

In \cite[Section 3.6]{sraco}, we have introduced a functor $\bullet_{\dagger,W'}: \HC(H_{\param}(W),\psi)\rightarrow
\HC^\Xi(H_{\param}(W'),\psi)$. Here we are going to explain a construction of this functor that is  equivalent to
but simpler than the one given in \cite{sraco}.

Set $Y:=\{b\in \h| W_b=W'\}, \h^{reg-W'}=\{b\in \h| W_b\subset W'\}$ so that $Y\subset \h^{reg-W'}/W'$
is closed and $\h^{reg-W'}\subset \h$ is a principal open subset. Set $$H_{\param}^{\wedge_Y}:=\C[\h^{reg-W'}/W']^{\wedge_Y}\otimes_{\C[\h/W]}H_{\param},$$
where $\C[\h^{reg-W'}/W']^{\wedge_Y}$ is the usual completion along a closed  subvariety,
note that this algebra is \'{e}tale over $\C[\h/W]$.
The space $H_{\param}^{\wedge_Y}$ is easily seen to be an algebra and this algebra is
filtered. Moreover, the group $\Xi$ acts on $H_{\param}^{\wedge_Y}$ by filtration preserving algebra
isomorphisms. So we can introduce a notion of a $\Xi$-equivariant HC $H_{\param}^{\wedge_Y}$-bimodule:
it is a $\Xi$-equivariant $H_{\param}^{\wedge_Y}$-bimodule $M'$ equipped with a filtration
such that $\gr M'$ is a finitely generated $\C[\h^{reg-W'}/W']^{\wedge_Y}\otimes_{\C[\h/W]}Z_{\param}(W)$-module.
We have a functor $\mathcal{F}': \HC(H_{\param},\psi)\rightarrow \HC^\Xi(H_{\param}^{\wedge_Y},\psi)$
given by $M\mapsto \C[\h^{reg-W'}/W']^{\wedge_Y}\otimes_{\C[\h/W]}M$. That the space 
$\mathcal{F}'(M)$ is an $H_{\param}^{\wedge_Y}$-bimodule is checked similarly to
\cite[Section 3.6]{sraco}. 

On the other hand, we can form the algebra $H_{\param}(W',\h)^{\wedge_Y}:= \C[\h^{reg-W'}/W']^{\wedge_Y}\otimes_{\C[\h/W']} H_{\param}(W',\h)$. The algebra $H_{\param}(W',\h)^{\wedge_Y}$ is  filtered. Further,  form the  centralizer
algebra $Z(W,W', H_{\param}(W',\h)^{\wedge_Y})$ from \cite[Section 3.2]{BE}. The group $\Xi$ acts on $Z(W,W', H_{\param}(W',\h)^{\wedge_Y})$ as explained in \cite[Section 2.3]{sraco}. 
Then, similarly to \cite[Section 3.2]{BE}, there is a filtration preserving isomorphism  \begin{equation}\label{eq:compl_iso}H_\param^{\wedge_Y}\cong Z(W,W', H_{\param}(W',\h)^{\wedge_Y})\end{equation} that
coincides with a natural isomorphism $$\C[\h^{reg-W'}/W']^{\wedge_Y}\otimes_{\C[\h/W]}(\C[\h]\#W)\cong Z(W,W', \C[\h^{reg-W'}]^{\wedge_Y}\#W')$$ on the filtration zero components, similarly to \cite[Section 3.3]{BE}. The isomorphism (\ref{eq:compl_iso}) is $\C[\param]$-linear and $\Xi$-equivariant. It is induced by an isomorphism from \cite[Section 2.13]{sraco} by passing to $\C^\times$-finite elements and then taking the quotient by $\hbar-1$, compare with \cite[Section 2.3]{fun_iso}. Note that the isomorphism we use does not need to be given by formulas in \cite[Section 3.3]{BE}.

The isomorphism (\ref{eq:compl_iso})
gives rise to an equivalence $$\HC^\Xi(H_{\param}^{\wedge_Y},\psi)\xrightarrow{\sim}
\HC^\Xi(H_{\param}(W',\h)^{\wedge_Y},\psi)$$ given by the push-forward under (\ref{eq:compl_iso}) followed by the
multiplication by a suitable $\Xi$-invariant primitive idempotent $e(W')\in Z(W,W',\C W')$, compare with Section \ref{SS_ind_res_O}. Let $\mathcal{F}:\HC(H_{\param},\psi)\xrightarrow{\sim} \HC^\Xi(H_{\param}(W',\h)^{\wedge_Y},\psi)$ be the resulting functor, it is exact.

On the other hand, we have a functor $\mathcal{G}: \HC^\Xi(H_{\param}(W'),\psi)\rightarrow
\HC^\Xi(H_{\param}(W',\h)^{\wedge_Y},\psi)$ given by $N\mapsto \C[Y\times \h_{W'}/W']^{\wedge_Y}\otimes_{\C[Y\times \h_{W'}/W']}(D(Y)\otimes N)$.

\begin{Lem}\label{Lem:restr_fun_defi}
The functor $\mathcal{G}$ is a full embedding whose image contains that of $\mathcal{F}$.
The functor $\mathcal{G}^{-1}\circ \mathcal{F}: \HC(H_{\param},\psi)\rightarrow \HC^\Xi(H_{\param}(W'),\psi)$
  coincides with the functor $\bullet_{\dagger}$ from \cite{sraco}.  
\end{Lem}
\begin{proof}
Let us show that the functor $\mathcal{G}$ is a full embedding by producing a left inverse functor.
First, take the centralizer of $D(Y)$ in $\mathcal{G}(N)$. The result is $\C[\h_{W'}/W']^{\wedge_0}\otimes_{\C[\h_{W'}/W']}N$. Then take the elements that are locally
finite for the Euler element $h'\in H_{\param}(W')$. The resulting bimodule is $N$. So we have
constructed a left inverse functor for $\mathcal{G}$.

The remaining two claims are proved simultaneously. In \cite[Section 3.6]{sraco} the functor $\bullet_{\dagger}$
was constructed as follows. Pick $M\in \HC(H_{\param},\psi)$, equip it with a good filtration
and form the Rees bimodule $M_\hbar$. Then we complete the bimodule $M_{\hbar}$ with respect
to the symplectic leaf $\Leaf_{W'}$ corresponding to $W'$, this leaf is given by $\{(x,y)\in \h\oplus \h^*| W_{(x,y)}=W'\}/\Xi$. The corresponding completion $R_\hbar(H_{\param})^{\wedge_{\Leaf_{W'}}}$
was shown in \cite[Section 2]{sraco} to be isomorphic to a {\it twist} of
$$Z\left(W,W', R_\hbar(H_{\param}(W',\h))^{\wedge_{\tilde{\Leaf}_{W'}}}\right)^{\Xi},$$
where we write $\tilde{\Leaf}_{W'}$ for $\{(x,y)\in \h\oplus \h^*| W_{(x,y)}=W'\}$.
By a twist we mean a sheaf whose sections on open affine subsets are the same but gluing
maps are different. In order to construct $\bullet_{\dagger,W'}$, we first lift $M_\hbar^{\wedge_{\Leaf_{W'}}}$
(viewed as a sheaf on $\Leaf_{W'}$) to a sheaf on $\tilde{\Leaf}_{W'}$, the resulting sheaf is given by
$$Z\left(W,W', R_\hbar(H_{\param}(W',\h))^{\wedge_{\tilde{\Leaf}_{W'}}}\right)^{tw}\otimes_{Z\left(W,W', R_\hbar(H_{\param}(W',\h))^{\wedge_{\tilde{\Leaf}_{W'}}}\right)^{tw,\Xi}}M_\hbar^{\wedge_{\Leaf_{W'}}}.$$
 Second, we untwist the lift getting a bimodule over
the sheaf $Z\left(W,W', R_\hbar(H_{\param}(W',\h)\right)^{\wedge_{\tilde{\Leaf}_{W'}}})$. Applying $e(W')$,
we get a  $R_\hbar(H_{\param}(W',\h))^{\wedge_{\tilde{\Leaf}_{W'}}}$-bimodule, say $\mathcal{N}_\hbar$. Next, we
show that the localization of the homogenized Weyl algebra of $(\h\oplus \h^*)^{W'}$
to $\tilde{\Leaf}_{W'}$ splits as a tensor factor of $\mathcal{N}_\hbar$. We take the centralizer of this localization
in $\mathcal{N}_\hbar$ getting a $R_\hbar(H_{\param}(W'))^{\wedge_0}$-bimodule $\mathcal{N}_\hbar^0$, take the finite vectors
(for the $\C^\times$-action induced by the dilations on $\h\oplus \h^*$)
in $\mathcal{N}_\hbar^0$ and mod out $\hbar-1$. The resulting bimodule is $\M_{\dagger,W'}$.

Thanks to that construction, what we need to show is that   $R_\hbar(\mathcal{F}(M))$ coincides  with
the $\C^\times$-finite part of  the global sections of the lift of $M_\hbar^{\wedge_{\Leaf_{W'}}}$  to the open subset
$Y\times \h^{*W'}\subset\tilde{\Leaf}_{W'}$ (here we consider the $\C^\times$-action that is trivial on $\h^*$
and by the dilations on $\h$). The global sections of interest is nothing else but
$\C[\h^{reg-W'}/W']^{\wedge_Y}\widehat{\otimes}_{\C[\h/W]}M_{\hbar}^{\wedge_\hbar}$ and we need to show that the
$\C^\times$-finite part coincides with $\C[\h^{reg-W'}/W']^{\wedge_Y}\otimes_{\C[\h/W]}M_{\hbar}$.
Recall that $\gr M$ is finitely generated as a $\C[\h]^W\otimes S(\h)^W\otimes \C[\param]$-module. Since both $\h$ and $\param$ have degree $1$, it follows that
the degree $n$ part in $M_\hbar^{\wedge_\hbar}$ is finitely generated over $\C[\h]^W$. The coincidence we need easily follows.
\end{proof}


\subsection{Restriction functors for HC bimodules: properties}\label{SS_dag_prop}
Let us quote some properties of the functor $\bullet_{\dagger,W'}$  established mostly in \cite{sraco}.

1)  The functor $\bullet_{\dagger,W'}$ is exact and $\C[\param]$-linear. This follows directly from the definition.

2) The functor $\bullet_{\dagger,W'}$ intertwines the tensor product functors.

3)  It is known (and easy to show) that the associated variety of a HC bimodule is
a union of the symplectic leaves in $(\h\oplus \h^*)/W$.  On the level of associated
varieties the functor $\bullet_{\dagger,W'}$ behaves as follows. Let the associated variety of
$M\in \HC(H_{\param},\psi)$ be the union of the leaves $\Leaf_{W_i}$ corresponding to the conjugacy classes of the parabolic subgroups
$W_1,\ldots,W_k$. Then the associated variety of $M_{\dagger,W'}$ is the union of the leaves
corresponding to all parabolic subgroups in $W'$ conjugate to one of $W_i$. This is established in
\cite[Proposition 3.6.5]{sraco}.

4) For a HC bimodule $M\in \HC(H_c,\psi)$, we can define its generic rank to be the generic rank of
$e\gr M$, where the associated graded is taken with respect to any good filtration.
We claim that $\bullet_{\dagger,W'}$ preserves the generic ranks (or sends a bimodule
to $0$). Indeed, the completion of $e'\gr(M_{\dagger,W'})$ at $0$ coincides with
the restriction of $e\gr(M)$ to a formal slice to $\mathcal{L}_{W'}$.  This coincidence (that 
follows from the construction of $\bullet_{\dagger,W'}$ above) gives the equality of generic ranks.

5) We can define the functor $\bullet_{\dagger,W'}$ for the HC bimodules over the spherical algebras
because (\ref{eq:compl_iso}) induces an isomorphism $eH_\param^{\wedge_Y}e\cong \left(e_{W'}H_{\param}(W',\h))^{\wedge_Y}e_{W'}\right)^{tw}$.
It can easily be seen from the definition that $(e\B e)_{\dagger,W'}\cong e_{W'}(\B_{\dagger,W'})e_{W'}$,
where $e_{W'}$ denotes the averaging idempotent in $\C W'$.

6) Consider the categories $\HC_{\partial \Leaf_{W'}}(H_{\param},\psi)\subset\HC_{\overline{\Leaf}_{W'}}(H_{\param},\psi)$
consisting of all HC bimodules $\B$ with $\VA(\B)$ contained in the boundary of $\Leaf_{W'}$ and in
the closure of $\Leaf_{W'}$, respectively. Let $\HC_{\Leaf_{W'}}(H_{\param},\psi)$ denote the quotient
category. Then $\bullet_{\dagger,W'}$ gives rise to a well-defined functor $\bullet_{\dagger,W'}:
\HC_{\Leaf_{W'}}(H_{\param},\psi)\rightarrow \HC_0^{\Xi}(H_{\param}(W'),\psi)$ that is a full embedding
with image closed under taking subquotients. This is a part of \cite[Theorem 3.4.5]{sraco}.

7) The functor $$\bullet_{\dagger,W'}: \HC_{\overline{\Leaf}_{W'}}(H_{\param},\psi)
\rightarrow \HC^\Xi_{0}(H_{\param}(W'),\psi)$$ admits a right adjoint  to be denoted
by $\bullet^{\dagger,W'}$. Both kernel and cokernel of the adjunction functor morphism $M\rightarrow
(M_{\dagger,W'})^{\dagger,W'}$ are supported on $\partial \Leaf_{W'}$. This is also a part of
\cite[Theorem 3.4.5]{sraco}.

8) We have a natural isomorphism  $\B_{\dagger,W'}\otimes_{H_{\param}(W')}\Res_W^{W'}(M)\xrightarrow{\sim}
\Res_{W}^{W'}(\B\otimes_{H_{\param}}M)$ for $\B\in \HC(H_{\param},\psi), M\in \OCat_{\param}$.
This is an easy consequence of \cite[Section 5.5]{sraco}.

\subsection{Tor's and Ext's}\label{SS_Tor_Ext}
Here we will investigate various Tor's and Ext's involving HC bimodules.

\begin{Prop}\label{Prop:HC_Ext}
Let $\mathcal{B}_1,\mathcal{B}_2$ be  HC $H_{\param}$-bimodules and $N\in \OCat_{\param}$. Then $\operatorname{Ext}^i_{H_{\param}}(\mathcal{B}_1,\mathcal{B}_2)$ and $
\operatorname{Tor}_i^{H_\param}(\mathcal{B}_1,\mathcal{B}_2)$ are HC bimodules, while $\operatorname{Ext}^i_{H_{\param}}(\mathcal{B}_1,N), \operatorname{Tor}_i^{H_\param}(\mathcal{B}_1,N)$
are in $\OCat_{\param}$.
\end{Prop}
\begin{proof}
The proofs of the claims involving HC bimodules  are similar, we will do the case of $\operatorname{Ext}^i_{H_{\param}}(\mathcal{B}_1,\mathcal{B}_2)$.

Let us equip $\mathcal{B}_1,\mathcal{B}_2$ with good filtrations. Let $H_{\param,\hbar}$
denote the Rees algebra, and let $\mathcal{B}_{1,\hbar},\mathcal{B}_{2,\hbar}$ be the Rees
bimodules. Then $\Ext^i_{H_{\param,\hbar}}(\mathcal{B}_{1,\hbar},\mathcal{B}_{2,\hbar})$
is a finitely generated graded $H_{\param,\hbar}$-bimodule. Moreover,
$\Ext^i_{H_{\param,\hbar}}(\mathcal{B}_{1,\hbar},\mathcal{B}_{2,\hbar})/(\hbar-1)=
\Ext^i_{H_\param}(\mathcal{B}_1,\mathcal{B}_2)$. So it remains to prove that,
for an element $a\in H_{\param,\hbar}$ that lies in $Z_{\param}$ modulo $\hbar$,
the operator $[a,\bullet]$ maps $\Ext^i_{H_{\param,\hbar}}(\mathcal{B}_{1,\hbar},\mathcal{B}_{2,\hbar})$
to $\hbar\Ext^i_{H_{\param,\hbar}}(\mathcal{B}_{1,\hbar},\mathcal{B}_{2,\hbar})$.
We have an exact sequence $$\Ext^i_{H_{\param,\hbar}}(\mathcal{B}_{1,\hbar},\mathcal{B}_{2,\hbar})\xrightarrow{\hbar}
\Ext^i_{H_{\param,\hbar}}(\mathcal{B}_{1,\hbar},\mathcal{B}_{2,\hbar})\rightarrow
\Ext^i_{H_{\param,\hbar}}(\mathcal{B}_{1,\hbar},\mathcal{B}_{2,\hbar}/(\hbar))$$
The last term coincides with $\Ext^i_{\gr H_{\param}}(\mathcal{B}_{1,\hbar}/(\hbar), \mathcal{B}_{2,\hbar}/(\hbar))$ so the operator $[a,\bullet]$ acts trivially on that term.
This implies our claim.

The proofs of the claims involving category $\mathcal{O}$ are similar and are based on the observation that
the objects of the category $\OCat_{\param}$ are precisely the graded modules $M$
whose associated varieties (in $(\h\oplus \h^*)/W$) lie in $\h/W$.
\end{proof}

Proposition \ref{Prop:HC_Ext} immediately extends to HC bimodules in $\HC(H_{\param^1},\psi)$,
and Tor's/Ext's taken over $H_{\param^1}$ or $H_{\param^1+\psi}$.

Now let us investigate derived tensor products of Harish-Chandra bimodules with a projective generator $P_c$
of $\OCat_c$.

\begin{Lem}\label{Lem:tens_HC_proj}
We have $\operatorname{Tor}_i^{H_c}(\mathcal{B},P_c)=0$ for $i>0$.
\end{Lem}
\begin{proof}
Recall that $P_c$ is a direct summand in  $H_c^{\wedge_\h}:=\varprojlim_{n\rightarrow \infty} H_c/H_c\h^n$.
So it is enough to show that $\operatorname{Tor}_i^{H_c}(\mathcal{B},H_c^{\wedge_\h})=0$ for $i>0$.

Consider the category $\mathcal{C}$ of all finitely generated $S(\h^*)^W$-$H_c$-bimodules
such that the adjoint action of $S(\h^*)^W$ is locally nilpotent.  We will prove that $\operatorname{Tor}_i^{H_c}(\mathcal{B},H_c^{\wedge_\h})=0$ for $i>0$ and any $\B\in \mathcal{C}$. 
Let us note that every  bimodule in $\mathcal{C}$ is finitely generated over $S(\h^*)^W\otimes S(\h)$, 
where we consider the action of $S(\h^*)^W$ by left multiplications and the action of $S(\h)$ by 
right multiplications, the proof of this repeats that of \cite[Lemma 3.3,(ii)]{BEG}.

Assume that we already know that $\operatorname{Tor}_j^{H_c}(\mathcal{B}, H_c^{\wedge_\h})=0$ for $j=1,\ldots,i-1$
and any $\mathcal{B}\in \mathcal{C}$. Since $S(\h^*)^W\otimes H_c$ is Noetherian, we see that there is a finite filtration on $\mathcal{B}$ such that the successive quotients are generated by elements commuting with $S(\h^*)^W$. So it is enough to prove that $\operatorname{Tor}_i^{H_c}(\mathcal{B}, H_c^{\wedge_\h})=0$ for  $\mathcal{B}\in
\mathcal{C}$ generated by elements commuting with $S(\h^*)^W$. We have an epimorphism $H_c^{\oplus k}\twoheadrightarrow \mathcal{B}$ of $S(\h^*)^W$-$H_c$-bimodules, let $K$ denote the kernel. Of course, $K$ is still in $\mathcal{C}$. 
Then we have an exact sequence
$\operatorname{Tor}^{H_c}_i(H_c^{\oplus k},H_c^{\wedge_\h})\rightarrow \operatorname{Tor}_i^{H_c}(\mathcal{B}, H_c^{\wedge_\h})\rightarrow \operatorname{Tor}_{i-1}^{H_c}(K, H_c^{\wedge_\h})$. If $i>1$, we are done by the inductive assumption.

Let us consider the case $i=1$.  It is enough to show that the functor $\bullet^{\wedge_\h}:=\bullet\otimes_{H_c}H_c^{\wedge_\h}$
is exact on $\mathcal{C}$. This functor coincides with the $\h$-adic completion on the right
that is exact on the even larger category of finitely generated $S(\h^*)^W\otimes S(\h)$-modules
by  standard results in  Commutative algebra.
\end{proof}

\begin{Lem}\label{Lem:tens_HC_prgen}
Let $M$ be a HC $H_{c'}$-$H_c$-bimodule. If $M\otimes_{H_c}P_c=0$, then $M=0$.
\end{Lem}
\begin{proof}
Assume $M\neq 0$. Recall that there is a parabolic subgroup $W'\subset W$ such that
$M_{\dagger,W'}$ is a nonzero finite dimensional bimodule. So there is a finite dimensional
$H_c(W')$-module $L'$ such that $M_{\dagger,W'}\otimes_{H_c(W')}L'\neq \{0\}$. Note that
${}^{\OCat}\operatorname{Res}_{W}^{W'}P_c$ is a projective generator of $\OCat_c(W')$.
That the module is projective is a consequence of the existence of a biadjoint functor
to ${}^{\OCat}\operatorname{Res}_{W}^{W'}$.
As was checked in \cite[Proposition 2.7]{shanvasserot}, the induction functor ${}^{\OCat}\operatorname{Ind}_{W'}^W$
does not annihilate any nonzero module. So ${}^{\OCat}\operatorname{Res}_{W}^{W'}P_c$
is a generator. We deduce that $M_{\dagger,W'}\otimes_{H_c(W')}{}^{\OCat}\operatorname{Res}_{W}^{W'}P_c\neq \{0\}$.
But the left hand side is ${}^{\OCat}\operatorname{Res}_{W}^{W'}(M\otimes_{H_c}P_c)\neq \{0\}$.
This contradiction finishes the proof.
\end{proof}

Finally, let us investigate the compatibility of Tor's with the restriction functors.

\begin{Lem}\label{Lem:dag_Tor}
We have  natural isomorphisms
$$\operatorname{Tor}^{H_{\param}}_i(\mathcal{B}^1,\mathcal{B}^2)_{\dagger,W'}\cong
\operatorname{Tor}^{H_{\param}(W')}_i(\B^1_{\dagger,W'}, \B^2_{\dagger,W'}),
\operatorname{Ext}_{H_{\param}}^i(\mathcal{B}^1,\mathcal{B}^2)_{\dagger,W'}\cong
\operatorname{Ext}_{H_{\param}(W')}^i(\B^1_{\dagger,W'}, \B^2_{\dagger,W'}).$$
\end{Lem}
\begin{proof}
In the notation of Lemma \ref{Lem:restr_fun_defi}, it suffices to show that the functors
$\mathcal{F},\mathcal{G}$ intertwine Tor's and Ext's. For $\mathcal{F}$, this follows from the observation
that  $\C[\h^{reg-W'}/W']^{\wedge_Y}$ is a flat $\C[\h/W]$-module. The proof for $\mathcal{G}$
is similar.
\end{proof}

Similarly, we see that ${}^{\OCat}\Res_{W}^{W'}(\operatorname{Tor}^{H_{\param}}_i(\mathcal{B},M))\cong
\operatorname{Tor}_i^{H_{\param}(W')}(\mathcal{B}_{\dagger,W'}, {}^{\OCat}\Res_{W}^{W'}(M))$.

\subsection{Relation to quantized quiver varieties}\label{SS_bimod_rk1}
Here we deal with the case when $W$ is a cyclic group. We will need an interpretation of the spherical subalgebras
$eH_ce$ as quantized quiver varieties due to Holland, \cite{Holland}, and some constructions and results from
\cite{BL}.  The results of this subsection will be used in
Section \ref{SS_Ringel_proof}.

Let $W=\Z/\ell \Z$.  Consider the space $R:=\C^\ell$ and the group $G:=(\C^\times)^\ell$ acting on
$R$ via $(t_1,\ldots,t_\ell).(x_1,\ldots,x_\ell)=(t_2x_1t_1^{-1},t_3x_2 t_2^{-1},\ldots, t_1 x_\ell t_\ell^{-1})$.
The induced action of $G$ on $T^*R=R\oplus R^*$ is Hamiltonian with moment map $\mu((x_i,y_i)_{i=1}^\ell)=(x_2y_2-x_1y_1, x_3y_3-x_2y_2,\ldots, x_1y_1-x_\ell y_\ell)$. It is easy to see that $\mu^{-1}(0)\quo G$ is identified with $\C^2/W$.

There is a quantum analog of this isomorphism originally due to Holland. Consider the Weyl algebra $\mathbb{A}(R\oplus R^*)$. We have a (symmetrized) quantum comoment map $\Phi: \g\rightarrow \mathbb{A}(R\oplus R^*)$ given by
$\epsilon_i\mapsto \frac{1}{2}(x_{i+1}y_{i+1}+y_{i+1}x_{i+1}-x_iy_i-y_ix_i)$, where $\epsilon_i, i=1,\ldots,\ell,$
is an element of the tautological basis in $\g=\C^n$. Then, for $\lambda\in \g^*$, we can form the quantum Hamiltonian reduction
$\A_\lambda:=[\mathbb{A}(R\oplus R^*)/\mathbb{A}(R\oplus R^*)\{x-\langle\lambda,x\rangle\}]^G$. This is a filtered
algebra (we consider the filtration by the order of a differential operator) with
$\gr\A_\lambda=\C[\h\oplus \h^*]^W$. Here $\lambda$ is recovered from $c$ by the following formulas
$$\sum_{i=1}^\ell \lambda_i=0, \lambda_i=\frac{1}{\ell}(1-2\sum_{j=1}^{\ell-1}c_j \exp(2\pi\sqrt{-1}ij/\ell)), i=1,\ldots,\ell-1.$$

We will also need resolutions of singularities of $\mu^{-1}(0)\quo G$ and their quantizations. Pick
$\theta\in \Z^{\ell}\cong \Hom(G,\C^\times)$ that satisfies $\sum_{i=1}^\ell \theta_i=0, \theta_i\neq \theta_j$
for $i\neq j$. Then we can form the $\theta$-semistable locus $(T^*R)^{\theta-ss}$ and the corresponding
GIT reduction $X^\theta:=\mu^{-1}(0)^{\theta-ss}\quo G$. The variety $X^\theta$ is a smooth symplectic variety
(in fact, independent of $\theta$ up to an isomorphism) with a resolution of singularities
morphism $\rho:X^\theta\rightarrow \C^2/W$. This variety can be quantized by a microlocal sheaf
of algebras, $\A_\lambda^\theta$, that is also constructed by quantum Hamiltonian reduction.
We  microlocalize $\mathbb{A}(T^*R)$ to a sheaf in conical topology on $T^*R$ so that the
restriction $\mathbb{A}(T^*R)|_{(T^*R)^{\theta-ss}}$ makes sense. Then we set
$\A_\lambda^\theta:=[\mathbb{A}(T^*R)|_{(T^*R)^{\theta-ss}}/\mathbb{A}(T^*R)|_{(T^*R)^{\theta-ss}}\{x-\langle \lambda,x\rangle\}]^G$,
this is a sheaf of filtered algebras on $X^\theta$ in conical topology with $\gr\A_\lambda^\theta=\mathcal{O}_{X^\theta}$.
We have $\Gamma(\A_\lambda^\theta)=\A_\lambda$, while the higher cohomology groups of $\A_\lambda^\theta$ vanish.

We can consider the category $\A_\lambda^\theta\operatorname{-Mod}$ of the quasi-coherent $\A_\lambda^\theta$-modules.
Then we have the global section functor $\Gamma:\A_\lambda^\theta\operatorname{-Mod}\rightarrow
\A_\lambda\operatorname{-Mod}$. There is a criterium for this functor to be an equivalence,
see \cite{Boyarchenko} (the formalism of $\Z$-algebras used in Boyarchenko's paper is equivalent
to the formalism we use by \cite[Section 5.3]{BPW}).
Namely, let us consider the permutation $\sigma$ of $\{1,2,\ldots,\ell\}$ such that
$\theta_{\sigma(1)}>\theta_{\sigma(2)}>\ldots>\theta_{\sigma(\ell)}$. Then the functor $\Gamma$
is an equivalence if and only if $\lambda_{\sigma(i)}-\lambda_{\sigma(j)}\not\in \Z_{\leqslant 0}$
for $i<j$. If $\Gamma:\A_\lambda^\theta\operatorname{-Mod}\rightarrow \A_\lambda\operatorname{-Mod}$
is an equivalence, then we say that $(\lambda,\theta)$ satisfies the abelian localization.

Let us now construct  some HC bimodules. Let $\varphi$ be a character of $G$.
We  consider the $\A_{\lambda+\varphi}^\theta$-$\A_\lambda^\theta$-bimodule
$$\A_{\lambda,\varphi}^\theta:=[\mathbb{A}(T^*R)|_{(T^*R)^{\theta-ss}}/\mathbb{A}(T^*R)|_{(T^*R)^{\theta-ss}}\{x-\langle \lambda,x\rangle\}]^{G,\varphi},$$ where the superscript $G,\varphi$ indicates that we take $(G,\varphi)$-semiinvariants.
Also we consider the $\A_{\lambda+\varphi}$-$\A_\lambda$-bimodule $\A_{\lambda,\varphi}^{(\theta)}:=\Gamma(\A^\theta_{\lambda,\chi})$.

We need to realize the inverse Ringel duality functor $R^{-1}$ as $\A^{(\theta)}_{\lambda,\varphi}\otimes^L_{\A_\lambda}\bullet$.
Namely, assume that $(\lambda,-\theta),(\lambda+\varphi,\theta)$ satisfy the abelian localization. This implies,
in particular, that the algebras $\A_\lambda,\A_{\lambda+\varphi}$ have finite homological dimension, hence
the corresponding Cherednik parameters are spherical, see \cite[Theorem 5.5]{etingof_affine}.
So it makes sense to speak about the categories $\OCat$
for $\A_\lambda,\A_{\lambda+\varphi}$, those categories are highest weight.

\begin{Lem}\label{Lem:ring_rk1}
There is an equivalence $\OCat(\A_{\lambda+\varphi})^\vee\cong \OCat(\A_\lambda)$ such that
$R^{-1}:D^b(\OCat(\A_\lambda))\xrightarrow{\sim} D^b(\OCat(\A_{\lambda+\varphi}))$ gets
identified with $\A^{(\theta)}_{\lambda,\varphi}\otimes^L_{\A_\lambda}\bullet$.
\end{Lem}
\begin{proof}
%
%
Set $\Cat^1:=\OCat(\A_\lambda), \Cat'^1:=\OCat(\A_{\lambda+\varphi})^\vee, \Cat^2:=\OCat(\A_{\lambda+\varphi})$.
We have perverse equivalences $\varphi:=\A^{(\theta)}_{\lambda,\varphi}\otimes^L_{\A_\lambda}\bullet:
D^b(\Cat^1)\xrightarrow{\sim}D^b(\Cat^2), \varphi':=R^{-1}:D^b(\Cat'^1)\xrightarrow{\sim} D^b(\Cat^2)$
with respect to the filtrations $\Cat^1_i, \Cat'^1_i, \Cat^2_i$ by dimension of support.
The filtration components in this case are zero for $i=2$ and are the subcategories of all finite dimensional
modules for $i=1$.  The perversity of $\varphi'$ was established in Section \ref{SS_Ringel}.
The construction of $\A^{(\theta)}_{\lambda,\varphi}$ implies that 
$\A^{(\theta)}_{\lambda,\varphi}\otimes^L_{\A_\lambda}\bullet$ restricts to an
abelian equivalence $\Cat^1/\Cat^1_1\xrightarrow{\sim}\Cat^2/\Cat^2_1$.
The perversity of $\varphi$ now follows from \cite[Proposition 4.1]{BL}. Since the filtrations on
$\Cat^2$ for $\varphi,\varphi'$ coincide, the composition $\iota:=\varphi^{-1}\circ\varphi':
D^b(\Cat'^1)\xrightarrow{\sim} D^b(\Cat^1)$ preserves the $t$-structures. This seems
to be a well-known fact, but we were not able to find a reference, so we provide
the proof for reader's convenience.

We prove that $H_j(\iota M')=0, j\neq 0,$ for $M'\in \Cat'^1_i$ using the decreasing induction
on $i$. Suppose that we know $H_j(\iota'\underline{M}')=H_j(\iota^{-1}\underline{M})=0, j\neq 0,$
for any $\underline{M}'\in \Cat'^1_{i+1}, \underline{M}\in \Cat^1_{i+1}$
and want to show that $H_j(\iota M')=H_j(\iota^{-1}M)=0$  for $M'\in \Cat'^1_i,
M\in \Cat^1_i$. Note that since the filtrations on $\Cat^2$ making $\varphi,\varphi'$
perverse, coincide, we get  $H_j(\iota M')\in \Cat^1_{i+1}, H_j(\iota^{-1}M)\in \Cat'^1_{i+1}$
for $j\neq 0$. Let $j$ be the maximal positive index such that $H_{-j}(\iota M')\neq 0$
and let $I$ be an injective in $\Cat^1_{i+1}$. Then $\Hom_{D^b(\Cat'^1)}(\iota M'[j], I)=
\Hom_{D^b(\Cat^1)}(M'[j], \iota^{-1}I)$. Since $\iota^{-1}$ gives an abelian equivalence
$\Cat^1_{i+1}\xrightarrow{\sim}\Cat'^1_{i+1}$, $\iota^{-1}I$ is injective in $\Cat'^1_{i+1}$.
So we see that   $\Hom_{D^b(\Cat^1)}(M'[j], \iota^{-1}I)=0$. Since this is true for
any injective $I$, we see that $H_{-j}(\iota M')=0$. Similarly, we see that
$H_j(\iota M'), H_{-j}(\iota^{-1}M), H_{j}(\iota^{-1}M)$ vanish. This completes
the proof of the induction step and the proof of the lemma.
\end{proof}

\section{Ringel duality via HC bimodules}\label{S_Ringel_HC}
\subsection{Main result}
The goal of this section is to prove that the Ringel duality functor is realized as the derived tensor product
with a HC bimodule. More precisely, we are going to prove the following.

\begin{Thm}\label{Thm:Ring_dual_HC}
Let $c\in \param, \psi\in \underline{\param}_{\Z}$ be such that the parameters $c,c-\psi$ lie in opposite open chambers
and are spherical.
Then there is  a labeling preserving equivalence $\OCat_c\cong \OCat_{c-\psi}^\vee$ and a HC $H_{c-\psi}$-$H_c$-bimodule
$\B_c(\psi)$ such that $R^{-1}\cong \B_c(\psi)\otimes^L_{H_c}\bullet$.
\end{Thm}

\subsection{Improved equivalence theorem}
\label{SS_impr_equi}
Our goal here is to prove an improved version of Proposition \ref{Prop:Chered_equi} with condition
(iv) omitted. 

\begin{Prop}\label{Prop:Chered_equi_impr}
Suppose that $c,c'\in \param$ satisfy the following conditions:
\begin{itemize}
\item[(i)] $c-c'\in \param_\Z$.
\item[(ii)] $\mathsf{tw}(c'-c)=\operatorname{id}$.
\item[(iii)] The ordering $\leqslant^c$ refines $\leqslant^{c'}$.
\end{itemize}
Then there is an equivalence $\OCat_c\xrightarrow{\sim} \OCat_{c'}$ of highest weight
categories mapping $\Delta_c(\lambda)$ to $\Delta_{c'}(\lambda)$ that intertwines
the KZ functors $\OCat_c,\OCat_{c'}\twoheadrightarrow \underline{\mathcal{H}}_q\operatorname{-mod}$.
\end{Prop}

This proposition is proved in the remainder of the section.

Choose a generic line $\ell$ through $c$. Recall that we consider the algebra $\tilde{H}_c(W):=\C[[\hbar]]\otimes_{\C[\param]}H_{\param}$, where the homomorphism $\C[\param]\rightarrow \C[[\hbar]]$
is given by restricting to the formal neighborhood of $0$ in $\ell$. Form an analogous algebra $\tilde{H}_{c'}(W)$
(for the line $\ell+c'-c$). Let $\tilde{\OCat}_c,\tilde{\OCat}_{c'}$ be the corresponding categories
 $\mathcal{O}$ and let $B_\hbar,B'_\hbar$ be the endomorphism algebras (with opposite multiplications)
of  projective generators in the  categories $\tilde{\OCat}_c,\tilde{\OCat}_{c'}$.

Let us write $A_\hbar$ for $\End(P_{KZ,\hbar})^{opp}$, where $P_{KZ,\hbar}\in \tilde{\OCat}_c$
is the deformation of $P_{KZ}\in \OCat_c$. Note that $A_\hbar=\tilde{\mathcal{H}}_q(W)$.  

Define a projective $B_\hbar$-module $\hat{P}_\hbar$ as follows. Take all  $\lambda_1,\ldots,\lambda_k
\in \Irr(W)$ such that  $\operatorname{codim}_\h\VA(L_c(\lambda_i))\leqslant 1$. Set
$\hat{P}_\hbar:=\bigoplus_{i=1}^k P_\hbar(\lambda_i)$, where $P_\hbar(\lambda_i)$ stands for
the deformation of the projective $B$-module $P(\lambda_i)$ to an automatically projective
$B_\hbar$-module. Let us point out that if (iii) of Proposition \ref{Prop:Chered_equi} holds,
then  the indecomposable summands of $\hat{P}=\hat{P}_\hbar/\hbar \hat{P}_\hbar$ are precisely those of $P_{KZ}$.

Set $\hat{A}_\hbar:=\End_{B_\hbar}(\hat{P}_\hbar)^{opp}$ and let $\hat{\pi}_\hbar$ be the natural
quotient functor $B_\hbar\operatorname{-mod}\twoheadrightarrow \hat{A}_\hbar\operatorname{-mod}$.
We write $\hat{A}, \hat{\pi}$ for the specializations to $\hbar=0$.

The following lemma is a special case of \cite[Lemma 2.8]{RSVV},
a different proof is given in \cite[Section 8.1]{VV_proof}.

\begin{Lem}\label{Lem:0_faith}
The functor $\hat{\pi}$ is 0-faithful.
\end{Lem}

Set $\underline{P}_\hbar:=\hat{\pi}_\hbar(P_{KZ,\hbar})$ and let $\underline{\pi}_\hbar$
denote the functor $\Hom_{\hat{A}_\hbar}(\underline{P}_\hbar,\bullet):\hat{A}_\hbar\operatorname{-mod}
\rightarrow A_\hbar\operatorname{-mod}$. It follows
that $\pi_\hbar=\underline{\pi}_\hbar\circ  \hat{\pi}_\hbar$.

We construct the algebras $\hat{A}'_\hbar,A'_\hbar$ and the functors $\hat{\pi}'_\hbar: \B'_\hbar\operatorname{-mod}\rightarrow \hat{A}'_\hbar\operatorname{-mod},
\underline{\pi}'_\hbar: \hat{A}'_\hbar\operatorname{-mod}\rightarrow A'_\hbar\operatorname{-mod}$
similarly to the above. Note that by condition (i) of the proposition, $A_\hbar=A'_\hbar$.

\begin{Lem}\label{Lem:iso_images}
There are progenerators in $\hat{A}_\hbar\operatorname{-mod},
\hat{A}'_\hbar\operatorname{-mod}$ whose images under $\underline{\pi}_\hbar,\underline{\pi}'_\hbar$
are isomorphic.
\end{Lem}
Let us first explain how this lemma implies the claim of the proposition. After that
we will prove the lemma.

The functors $\underline{\pi}_\hbar, \underline{\pi}'_\hbar$ are fully faithful
on the projectives because the functors $\pi_\hbar, \pi'_\hbar$ are.
So Lemma \ref{Lem:iso_images} gives an equivalence
\begin{equation}\label{eq:equiv_interm_quot}\hat{A}_\hbar\operatorname{-mod}\xrightarrow{\sim}
\hat{A}'_\hbar\operatorname{-mod} \end{equation}
that intertwines the functors $$\underline{\pi}_\hbar:\hat{A}_\hbar\operatorname{-mod}
\twoheadrightarrow A_\hbar\operatorname{-mod}, \underline{\pi}'_\hbar: \hat{A}'_\hbar\operatorname{-mod}
\twoheadrightarrow A_\hbar\operatorname{-mod}.$$

We claim that the bijection
\begin{equation}\label{eq:bijection}\operatorname{Irr}(W)\xrightarrow{\sim} \operatorname{Irr}(\hat{A}_\hbar[\hbar^{-1}])\xrightarrow{\sim}
\operatorname{Irr}(\hat{A}'_\hbar[\hbar^{-1}])\xrightarrow{\sim} \operatorname{Irr}(W)\end{equation}
induced by the equivalence (\ref{eq:equiv_interm_quot}) is the identity.
Since the functors $\pi_\hbar,\pi_\hbar'$  factor through $\underline{\pi}_\hbar, \underline{\pi}'_\hbar$,
the bijection (\ref{eq:bijection}) coincides with
\begin{equation}\label{eq:bijection1}\operatorname{Irr}(W)\xrightarrow{\sim} \operatorname{Irr}(A_\hbar[\hbar^{-1}])\xrightarrow{\sim} \operatorname{Irr}(W).\end{equation}
By condition (ii) of the proposition combined with \cite[Theorem 7.11]{BC}, we see that
(\ref{eq:bijection1}) is the identity.

Now the existence of an equivalence $B_\hbar\operatorname{-mod}\xrightarrow{\sim} B'_\hbar\operatorname{-mod}$
that intertwines the functors $\pi_\hbar,\pi'_\hbar$ and preserves the labels follows from
Proposition \ref{Prop:abstr_equiv}. Indeed, $\hat{\pi}$ is 0-faithful by Lemma \ref{Lem:0_faith},
and $\hat{\pi}'$ is 0-faithful for the same reasons, which is (i) of Proposition \ref{Prop:abstr_equiv}.
(ii) of that proposition follows from (iii) of the present proposition combined with
the claim that (\ref{eq:bijection}) is the identity.

\begin{proof}[Proof of Lemma \ref{Lem:iso_images}]
We need to show that there are projective generators $\widehat{P}_\hbar$ of $\hat{A}_\hbar\operatorname{-mod}$
and $\widehat{P}'_\hbar$ of $\hat{A}_\hbar'\operatorname{-mod}$ such that $\underline{\pi}_\hbar(\widehat{P}_\hbar)\cong
\underline{\pi}'_\hbar(\widehat{P}'_\hbar)$.

First, let us deal with the case $\dim \h=1$. For $i,j\in \{1,\ldots,\ell_H\}$ we write $i\sim j$ if $q_{i}=q_j$
equivalently, $h_i-h_j+\frac{i-j}{\ell}\in \Z$. The category $\mathcal{H}_q(W)\operatorname{-mod}$ splits
into the sum of blocks, one per each equivalence class in $\{1,\ldots,\ell\}$. The labels $\lambda_i$
(we write $\lambda_i$ for the label corresponding to the character $z\mapsto z^{-i}$)
belonging to the same block have pairwise different values of the $c$-function with integral pairwise
differences (recall that the values of the $c$-function is $\ell h_i$, see the proof
of Lemma \ref{Lem:c_fun_int}).
So  they are ordered linearly in a highest weight order. This characterizes
the images of the projectives in $A_\hbar\operatorname{-mod}$  uniquely, see \cite[Section 7.2]{VV_proof}
or \cite[2.4.4]{RSVV}.

Now let us deal with the general case. Let us decorate the objects related to $W_H$ with the superscript ``$H$'', e.g.,
for a character $\lambda$ of $W_H$ let $c^H_\lambda$ denote the value of the $c$-function for $W_H$.
For $\widehat{P}_\hbar$ we take the sum of all objects of the form ${}^\mathcal{O}\Ind_{W_H}^W P^H_\hbar$
(or more precisely the image of this object under $\hat{\pi}_\hbar$), where $H$ runs  over the
conjugacy classes of the reflection hyperplanes and $P^H_\hbar$ is the sum of the indecomposable projectives
in the deformed category $\OCat_c(W_H)$. The object $\hat{P}_\hbar$ is a projective generator in $\hat{A}_\hbar\operatorname{-mod}$. Moreover,
$\pi_\hbar({}^\mathcal{O}\Ind_{W_H}^W P^H_\hbar)={}^\mathcal{H}\Ind_{W_H}^W(\pi^H_\hbar(P^H_\hbar))$
by what was recalled in Section \ref{SS_ind_res_O}. Note that, for two characters $\lambda,\mu$ of $W_H$ (that extend
to characters of $W$ as was recalled in Section \ref{SS_param}), we
have $c^H_\lambda-c^H_\mu=\frac{1}{N}(c_\lambda-c_\mu)$, where $N$ is the number of hyperplanes in the conjugacy
class of $H$. From the previous paragraph, it follows that $\pi^H_{\hbar}(P^H_\hbar)=\pi'^H_{\hbar}(P'^H_\hbar)$.
This completes the proof.
\end{proof}

The proof of Proposition \ref{Prop:Chered_equi_impr} is now complete.

\subsection{KZ vs Ringel duality}\label{SS_KZ_Ringel}
Pick a parameter $c$ in an open chamber. Let $\psi\in \underline{\param}_{\Z}$ be such that
$c-\psi$ lies in the opposite chamber. The main result of this subsection  is the following proposition.

\begin{Prop}\label{Prop:Ring_dual}
There is an identification $\OCat_{c-\psi}^\vee\cong \OCat_c$ such that the Ringel duality
$R^{-1}:D^b(\OCat_{c})\rightarrow D^b(\OCat_{c-\psi})$ maps $\Delta_c(\lambda)$
to $\nabla_{c-\psi}(\lambda)$ and  intertwines the KZ functors.
\end{Prop}

We will establish a labeling preserving  equivalence of $\OCat_{c-\psi}$ with $\OCat_{c}^{r,opp}$
preserving the associated varieties.
For this, we need to establish an analog of the KZ functor $\KZ^{ro}: \OCat_c^{r,opp}\rightarrow
\underline{\mathcal{H}}_q(W)\operatorname{-mod}$ (together with its deformed version, meaning deformations
over $\C[[\hbar]]$) and establish
its compatibility with the Ringel duality. Then we will use an argument that is completely analogous
to the previous subsection.

We set $\KZ^{ro}:=\KZ\circ D^{-1}$.

\begin{Lem}\label{Lem:KZ_ro}
For $M\in \OCat_c^{r,opp}$, the object $\KZ^{ro}(M)$ is in  homological degree $0$. The functor $\KZ^{ro}$ defines
an equivalence of $\OCat_c^{r,opp}/(\OCat_c^{r,opp})^{tor}$ and $\underline{\mathcal{H}}_q(W)\operatorname{-mod}$.
\end{Lem}
\begin{proof}
This is a direct consequence of Lemma \ref{Lem:D_perv} and the claim that $\KZ$ induces an equivalence
$\OCat_c/\OCat_c^{tor}\cong \underline{\mathcal{H}}_q(W)\operatorname{-mod}$.
\end{proof}

%
%
%
%
We still can consider the restriction functors ${}^{\OCat,r}\Res_W^{W'}$
for $\OCat_c^r(W)$. They are compatible with the homological
dualities in the following sense. Let $D_{W'}$ denote the homological duality $D^b(\OCat_c(W'))\rightarrow D^b(\OCat_c^{r}(W'))$.

\begin{Lem}\label{Lem:D_vs_Res}
We have a functor isomorphism $D_{W'}\circ {}^{\OCat}\Res_{W}^{W'}\cong {}^{\OCat,r}\Res_W^{W'}\circ D$.
\end{Lem}

Here is an immediate corollary of Lemma \ref{Lem:D_vs_Res}, the definition of $\KZ^{ro}$
and the claim that the usual restriction functors intertwine $\KZ$. Here
${\KZ^{ro}}'$ is the KZ functor $\OCat_c^{r}(W')^{opp}\twoheadrightarrow \mathcal{H}_q(W')\operatorname{-mod}$.

\begin{Cor}\label{Cor:KZro_vs_Res}
We have ${\KZ^{ro}}'\circ {}^{\OCat,r}\Res_{W}^{W'}\cong {}^{\mathcal{H}}\Res_W^{W'}\circ \KZ^{ro}$.
\end{Cor}

\begin{proof}[Proof of Lemma \ref{Lem:D_vs_Res}]
The homological duality commutes with the completion functor $M\mapsto M^{\wedge_b}$.
Indeed, $$R\operatorname{Hom}_{H_c}(M,H_c)\otimes_{H_c}H_c^{\wedge_b}\xrightarrow{\sim}
 R\operatorname{Hom}_{H_c}(M,H^{\wedge_b}_c) \xrightarrow{\sim} R\operatorname{Hom}_{H_c^{\wedge_b}}(M^{\wedge_b}, H_c^{\wedge_b})$$
because $H_c^{\wedge_b}$ is flat as a right and as a left $H_c$-module.
Also the homological duality obviously commutes with $\theta_{b*}$ and with multiplying by $e(W')$.
So the functor $M\mapsto e(W')\theta_{b*}(M^{\wedge_b})$ commutes with the homological duality.
On the other hand, the functor $N\mapsto (\C[\h^{W'}]\otimes N)^{\wedge_0}$  commutes with the homological
dualities as well. The latter functor is an equivalence of $\OCat_c(W')$ and the category
$\OCat(H_c^{\wedge_0}(W',\h))$ of all $H^{\wedge_0}(W',\h)$-modules that are finitely
generated over $\C[\h]^{\wedge_0}$, see \cite[Section 3.5]{BE}. Under the identification
$\OCat_c(W')\cong \OCat(H_c^{\wedge_0}(W',\h))$, the functor $\Res_W^{W'}$ becomes
$M\mapsto e(W')\theta_{b*}(M^{\wedge_b})$. As we have seen the latter commutes
with the homological dualities, and this proves the lemma.
\end{proof}

Let us make two remarks about the functors above. First of all, the functor ${}^{\OCat,r}\Res_{W}^{W'}$
admits a left and, simultaneously, a right adjoint functor, ${}^{\OCat,r}\Ind_{W'}^{W}$, the proof is the same as for the left-handed analogs. Second of all, we can define straightforward analogs of the functors above for the deformed categories,
for example, $\tilde{\KZ}^{ro}: \tilde{\OCat}_c^{r,opp}\twoheadrightarrow \tilde{\underline{\mathcal{H}}}_q(W)\operatorname{-mod}$
defined as $\tilde{\KZ}^{ro}:=\tilde{\KZ}\circ D^{-1}$.  Straightforward analogs of
Lemmas \ref{Lem:KZ_ro},\ref{Lem:D_vs_Res} and also of Corollary \ref{Cor:KZro_vs_Res}
still hold. The proofs of these analogs repeat the proofs above.

\begin{proof}[Proof of Proposition \ref{Prop:Ring_dual}]
We  adopt the argument of Section \ref{SS_impr_equi} to prove that there is
a label preserving highest weight equivalence $\OCat_c^{r,opp}\xrightarrow{\sim}
\OCat_{c-\psi}$. For this we note that the $c$-orders for both categories are opposite
to the $c$-order for $\OCat_c$. Also the identification of the labels in $\OCat_{c-\psi}(W)$
and $\OCat_c^r(W)^{opp}$ is the identity. This is because $\tilde{\KZ}(\tilde{\Delta}_{c-\psi}(\lambda))[\hbar^{-1}]\cong
\tilde{\KZ}(\tilde{\Delta}_c(\lambda))[\hbar^{-1}]$ 
and $\tilde{\KZ}^{ro}(\tilde{\nabla}_c^{ro}(\lambda))=
\tilde{\KZ}(\tilde{\Delta}_c(\lambda))$ (we use the notation like $\tilde{\Delta},\tilde{\nabla}$ for the standard
and costandard modules in the  categories $\tilde{\mathcal{O}}$). The former isomorphism holds because $\mathsf{tw}(\psi)=\operatorname{id}$,  the latter is a consequence of the definition of
$\tilde{\KZ}^{ro}$. So we see that the direct analogs of (i)-(iii) in Proposition \ref{Prop:Chered_equi}
hold.

We can now  get rid of (iv) of Proposition \ref{Prop:Chered_equi} as in the proof of Lemma \ref{Lem:iso_images}.
Namely, we first show that $\tilde{\OCat}_c^{r,opp}(W_H)\xrightarrow{\sim}
\tilde{\OCat}_{c-\psi}(W_H)$ (an equivalence intertwining the KZ functors)
for all reflection hyperplanes $H$. Thanks to (the deformed version of) Corollary \ref{Cor:KZro_vs_Res}, the deformed
$\KZ^{ro}$-functors intertwine (again deformed) induction functors.
This allows to establish an equivalence of the quotients of
$\tilde{\OCat}_c^{r,opp}(W), \tilde{\OCat}_{c-\psi}(W)$ defined
analogously to $\hat{A}_\hbar\operatorname{-mod}, \hat{A}'_\hbar\operatorname{-mod}$,
again in the same way as in Lemma \ref{Lem:iso_images}.

A labeling preserving
equivalence $\OCat_c^{r,opp}\cong \OCat_{c-\psi}$ is established, it intertwines the functors $\KZ^{ro}$
and $\KZ$.
\end{proof}

Let us also point out that the equivalence $\OCat_c^{r,opp}\cong \OCat_{c-\psi}$ preserves the supports.
This follows from Corollary \ref{Cor:KZro_vs_Res} combined with \cite[6.4.9]{GL}.

\subsection{Bimodule $\B_c(\psi)$}
Let $c$ be a parameter and $\psi\in \param_{\Z}$.

\begin{Lem}\label{Lem:B_def}
There is a unique  HC $H_{c-\psi}$-$H_c$-bimodule $\mathcal{B}_c(\psi)$ with the following properties:
\begin{enumerate}
\item $\mathcal{B}_{c}(\psi)$ is simple.
\item $\mathcal{B}_{c}(\psi)[\delta^{-1}]$ is the regular $D(\h^{reg})\#W$-bimodule.
\end{enumerate}
\end{Lem}
\begin{proof}
Let us recall that, for a character $\chi$ of $W$, we have the HC $H_{c'+\bar{\chi}}$-$H_{c'}$-bimodule
$\B_{c',\bar{\chi}}$ and the HC $H_{c'}$-$H_{c'+\bar{\chi}}$-bimodule $\B_{c'+\bar{\chi},-\bar{\chi}}$.
There is a sequence of characters $\chi_1,\ldots,\chi_k$ of $W$ such that $-\psi=\epsilon_1\bar{\chi}_1+\ldots+\epsilon_k\bar{\chi}_k$, where $\epsilon_i=\pm 1$,
and we set $c_i:=c+\sum_{j=1}^i \epsilon_j\bar{\chi}_j$.
Consider the $H_{c-\psi}$-$H_c$-bimodule  $$\B_{c,\psi}:=\B_{c_{k-1}, \epsilon_k\bar{\chi}_k}\otimes_{H_{c_{k-1}}}\ldots\otimes_{H_{c_1}}\B_{c_0,\epsilon_1\bar{\chi}_1}$$ (the notation $\B_{c,\psi}$ is ambiguous as the bimodule depends
on the choice $\chi_i,\epsilon_i$ but this is not important for us).
By the construction, $e\B_{c',\bar{\chi}}e[\delta^{-1}]=D(\h^{reg})^{W, \chi},
e\B_{c'+\bar{\chi},-\bar{\chi}}e[\delta^{-1}]=D(\h^{reg})^{W, \chi^{-1}}$. These bimodules
are isomorphic to $D(\h^{reg})^{W}$, isomorphisms are given by multiplying
by suitable products of elements $\alpha_s$.
Therefore $e \B_{c,\psi}e[\delta^{-1}]=D(\h^{reg})^W$.
So $\B_{c,\psi}[\delta^{-1}]=D(\h^{reg})\#W$. The algebra $D(\h^{reg})\#W$ is simple.
So there is a unique simple composition factor of $\B_{c,\psi}$ that does not vanish
under inverting $\delta$. We take this composition factor for $\B_c(\psi)$.

Let us prove the uniqueness of $\B_c(\psi)$. From (2) and the construction of $\bullet_{\dagger,\{1\}}$
recalled in Section \ref{SS_dag_constr} it follows that $\B_{c}(\psi)_{\dagger,\{1\}}$
is the trivial $W$-module, let us write $\mathbf{1}_{c,c-\psi}$ for this bimodule.
So we get a homomorphism $\B_c(\psi)\rightarrow {\bf 1}_{c,c-\psi}^{\dagger,\{1\}}$ whose kernel and cokernel
have proper associated varieties by 7) of Section \ref{SS_dag_prop}. Now
(1) and (2) determine $\B_{c}(\psi)$ uniquely.
\end{proof}

We will need an equivalent formulation of (2).

\begin{Lem}\label{Lem:loc_equiv}
Let $\B$ be a  HC $H_{c-\psi}$-$H_c$-bimodule. Then the following two conditions are equivalent:
\begin{enumerate}
\item $\B[\delta^{-1}]$ is a regular $D(\h^{reg})\#W$-bimodule.
\item The functor $\B\otimes_{H_c}\bullet:\OCat_c\rightarrow \OCat_{c-\psi}$ intertwines the KZ functors.
\end{enumerate}
\end{Lem}
\begin{proof}
In (2) we can replace the KZ functors with the localization functors. Now (1) obviously implies (2).
Let us prove the implication (2)$\Rightarrow$(1). We have $$e\mathcal{B}e[\delta^{-1}]\otimes_{D(\h^{reg})^W}e\mathsf{loc}\Delta_c(\operatorname{triv})\cong
e\mathsf{loc}\Delta_c(\operatorname{triv}).$$
Since the adjoint action of $\C[\h^{reg}]^W$ on $e\mathcal{B}e[\delta^{-1}]$ is locally nilpotent,
we see that the previous isomorphism gives rise to a $D(\h^{reg})^W$-bimodule homomorphism
$$e\mathcal{B}e[\delta^{-1}]\rightarrow \operatorname{Diff}(e\loc(\Delta_c(\operatorname{triv})),
e\loc(\Delta_c(\operatorname{triv}))),$$
where on the right hand side we have the space of differential maps.
But $e\loc(\Delta_c(\operatorname{triv})))=\mathcal{O}_{\h^{reg}/W}$ and so the space of the differential
maps we need is just $D(\h^{reg})^W$. Since this bimodule is simple, we conclude that
$e\mathcal{B}e[\delta^{-1}]\twoheadrightarrow D(\h^{reg})^W$. If $K$ is the kernel of this map,
then $K\otimes_{D(\h^{reg})^W}\mathcal{O}_{\h^{reg}/W}=0$.

Let us show that the last equality
implies $K=0$. Pick a point $x\in \h^{reg}/W$ and let $Z$ denote its formal
neighborhood in $\h^{reg}/W$. The restriction $K_Z:=\C[Z]\otimes_{\C[\h^{reg}/W]}K$
is a $D(Z)$-bimodule with $K_Z\otimes_{D(Z)}\C[Z]=0$. Note that $K_Z$ comes equipped
with a filtration compatible with the filtration on $D(Z)$ by the order of a
differential operator and such that $\gr K_Z$ is a $\C[T^*Z]$-module. It follows that
$K_Z$ is finitely generated over $\C[Z]\otimes \C[T^*_xZ]$ (where the first factor
acts by left multiplications and the right factor acts by right multiplications).
Therefore $K_Z$ is the direct sum of several copies of $D(Z)$. From
$K_Z\otimes_{D(Z)}\C[Z]=0$ we deduce that $K_Z=0$. It follows that $K=0$.
\end{proof}


Now let us describe  $\B_{c}(\psi)_{\dagger,W'}$, where $W'$ be a parabolic subgroup of $W$.

\begin{Prop}\label{Lem:B_socle}
We have $\B_{c}(\psi)_{\dagger,W'}=\B'_{c}(\psi)$, where the right hand side is the analog of $\B_c(\psi)$ for $W'$.
\end{Prop}
\begin{proof}
The proof is in three steps.
\begin{enumerate}
\item The socle of the bimodule $\B_{c}(\psi)_{\dagger,W'}$ is a simple HC bimodule whose associated variety
coincides with $(\h_{W'}\oplus \h_{W'}^*)/W'$.
\item The head of the bimodule $\B_{c}(\psi)_{\dagger,W'}$ is a simple HC bimodule whose associated variety
coincides with $(\h_{W'}\oplus \h_{W'}^*)/W'$.
\item The bimodule $\B_{c}(\psi)_{\dagger,W'}$ satisfies the analogs of (1) and (2) in Lemma \ref{Lem:B_def}.
\end{enumerate}

Let us show (1). It is enough to show that there are no HC bimodules with proper associated variety in the socle of
$\B_{c}(\psi)_{\dagger,W'}$. Indeed, then the socle is simple because the generic rank of
$\B_{c}(\psi)_{\dagger,W'}$ is $1$.

So assume the converse: there is a subbimodule in $\B_{c}(\psi)_{\dagger,W'}$ with proper associated
variety. So there is a parabolic subgroup $W''\subset W'$ such that the ideal $J'\subset \C[\h_+]^{W'}$
of the stratum corresponding to $W''$ has nonzero annihilator in $\B_{c}(\psi)_{\dagger,W'}$.
From the construction of $\bullet_{\dagger,W'}$ it follows that the ideal $J\subset \C[\h]^W$
corresponding to $W''$ has nonzero annihilator in $\C[\h^{reg-W'}/W']^{\wedge_Y}\otimes_{\C[\h/W]}\B_c(\psi)$.
But since $\C[\h^{reg-W'}/W']^{\wedge_Y}$ is a flat $\C[\h/W]$-algebra,
we see that \begin{align*}&\Hom_{\C[\h/W]}(\C[\h/W]/J, \C[\h^{reg-W'}/W']^{\wedge_Y}\otimes_{\C[\h/W]}\B_c(\psi))=\\
&\C[\h^{reg-W'}/W']^{\wedge_Y}\otimes_{\C[\h/W]}\Hom_{\C[\h/W]}(\C[\h/W]/J,\B_c(\psi)).\end{align*} So $J$
has nonzero annihilator in $\B_c(\psi)$ as well. The union of the annihilators of the ideals
$J^m, m\in \Z_{>0},$ in $\B_c(\psi)$ is a  subbimodule in $\B_{c}(\psi)$ that needs to coincide with
$\B_c(\psi)$ because the latter is simple. The associated variety
of $\B_c(\psi)$ is contained in $\overline{\Leaf}_{W''}$ because we have an epimorphism
$(H_{c-\psi}/J^k)^{\oplus r}\twoheadrightarrow \B_c(\psi)$ of left $H_{c-\psi}$-modules. Contradiction
with the fact that the associated variety of $\B_c(\psi)$ is $(\h\oplus \h^*)/W$.

Let us proceed to (2). Assume the contrary: there is an epimorphism $\B_c(\psi)_{\dagger,W'}\twoheadrightarrow \B'$,
where $\B'$ is a $H_{c-\psi}(W')$-$H_c(W')$-bimodule with proper associated variety. So, in the notation of
Lemma \ref{Lem:restr_fun_defi}, $\mathcal{F}'(\B_c(\psi))\twoheadrightarrow \iota\circ\mathcal{G}(\B')$
(an epimorphism of $H_{c-\psi}^{\wedge_Y}$-$H_c(\psi)^{\wedge_Y}$-bimodules), where $\iota$
stands for the Morita equivalence between the categories of $H_{?}^{\wedge_Y}$- and $H_?(W',\h)^{\wedge_Y}$-bimodules.
Note that the composition $\B_c(\psi)\rightarrow \iota\circ\mathcal{G}(\B')$ cannot be zero because
$\B_c(\psi)$ generates $\mathcal{F}(\B_c(\psi))$. On the other hand, $\iota\circ\mathcal{G}(\B')$
is annihilated by some nontrivial ideal in $\C[\h]^W$. Since $\B_c(\psi)$ is simple, it also
needs to be annihilated by that ideal. This contradiction proves (2).

It remains to show that $\B_c(\psi)_{\dagger,W'}$ satisfies (1) and (2) of Lemma \ref{Lem:B_def}.
(1) of that Lemma is a consequence of the observation that the generic
rank of $\B_{c}(\psi)$ is 1 combined with (1) and (2) of the present proof.
(2) follows from Lemma \ref{Lem:loc_equiv}, and the functor isomorphisms
$\operatorname{KZ}'\circ \,^{\mathcal{O}}\operatorname{Res}_W^{W'}\cong \,^{\mathcal{H}}\Res_W^{W'}
\circ \operatorname{KZ}$ and $\mathcal{B}_{\dagger,W'}\otimes_{H_c(W')}{}^{\OCat}\Res_W^{W'}(\bullet)\cong
{}^{\OCat}\Res_W^{W'}(\B_c(\psi)\otimes_{H_c}\bullet)$. More precisely, the functor isomorphisms
imply that the functor $\B_{c}(\psi)_{\dagger,W'}\otimes_{H_c(W')}\bullet$ intertwines
 $\operatorname{KZ}'$  on the essential images of $\Res_W^{W'}$. But, as we have seen in the proof of Lemma
\ref{Lem:tens_HC_prgen}, the essential image contains a projective generator. So
$\B_{c}(\psi)_{\dagger,W'}\otimes_{H_c(W')}\bullet$ intertwines the functors $\operatorname{KZ}'$.
\end{proof}

Now let us describe the bimodule $\B_c(\psi)_{\dagger,W_H}$ under the following assumptions
\begin{itemize}
\item The parameters $c,c-\psi$ (or, more precisely, their restrictions to $W_H\cap S$) lie in opposite open
chambers for $W_H$.
\item The parameters $c,c-\psi$ are spherical.
\end{itemize}

Recall an isomorphism $e_{W_H}H_ce_{W_H}\cong \A_\lambda$ from Section \ref{SS_bimod_rk1}. Let $\lambda^{opp}$ be the parameter corresponding to $c-\psi$ and $\theta$ be a stability condition such that $(\lambda^{opp},\theta)$ satisfies the abelian localization.

\begin{Prop}\label{Prop:rank1}
We have an isomorphism $e_{W_H}\B_c(\psi)_{\dagger,W_H}e_{W_H}\cong \A_{\lambda,\lambda^{opp}-\lambda}^{(\theta)}$.
\end{Prop}
\begin{proof}
Let $\B^H_c(\psi)$ be the analog of $\B_c(\psi)$ for $W_H$ so that $\B_c(\psi)_{\dagger,W_H}=\B^H_c(\psi)$.
It remains to show that $e_{W_H}B^H_c(\psi)e_{W_H}=\A_{\lambda,\lambda^{opp}-\lambda}^{(\theta)}$.
Let us notice that both sides are simple bimodules. For the right hand side, this is \cite[Lemma 5.1]{BL}
and for the left  hand side this follows from the definition. Also when we localize $\delta_H$, we will get
the regular $D(\C^\times)^{W_H}$-bimodules. For the left hand side, this again follows from the definition.
For the right hand side, this follows  from the  construction: $\A_{\lambda,\lambda^{opp}-\lambda}^\theta$
quantizes a line bundle on $X^\theta$ and the restriction of any line bundle to $(\C^\times\times \C)/W_H$
is trivial because the latter variety is factorial. The isomorphism $e_{W_H}B^H_c(\psi)e_{W_H}=\A_{\lambda,\lambda^{opp}-\lambda}^{(\theta)}$ now follows from the uniqueness part of Lemma \ref{Lem:B_def}.
\end{proof}

\subsection{Proof of Theorem \ref{Thm:Ring_dual_HC}}\label{SS_Ringel_proof}
First, we will need a corollary of Proposition \ref{Prop:Ring_dual}.
Assume that $c,c-\psi$ lie in the opposite chambers.  Let $T_{c-\psi}$ denote the sum of all indecomposable tiltings in
$\OCat_{c-\psi}$ and $P_{c}$ be the sum of all indecomposable projectives in $\OCat_{c}(W)$.
Recall the functor $\loc:\OCat_?\rightarrow \LS_{rs}(\h^{reg}/W)$.

\begin{Cor}\label{Cor:loc_coinc}
We have $\loc(P_{c}(\lambda))\cong \loc(T_{c-\psi}(\lambda))$.
\end{Cor}
\begin{proof}
Let us remark that the images of $\OCat_c,\OCat_{c-\psi}$ under $\loc$ coincide, for example,
because the Hecke algebras are the same.
Further, by Proposition \ref{Prop:Ring_dual}, the images of $\loc(P_{c}(\lambda)), \loc(T_{c-\psi}(\lambda))$
under the equivalence $\operatorname{im}\loc\cong \underline{\mathcal{H}}_q\operatorname{-mod}$  are isomorphic.
\end{proof}

\begin{proof}[Proof of Theorem \ref{Thm:Ring_dual_HC}]
Let us show that $\mathcal{B}_{c}(\psi)\otimes_{H_c}P_c\cong T_{c-\psi}$.

First of all, since $\mathcal{B}_c(\psi)[\delta^{-1}]=D(\h^{reg})\#W$ and $\loc(P_c)\cong
\loc(T_{c-\psi})$, we get an isomorphism $\loc(\mathcal{B}_c(\psi)\otimes_{H_c}P_c)\cong
\loc(T_{c-\psi})$. This gives rise to a homomorphism $\iota:\mathcal{B}_c(\psi)\otimes_{H_c}P_c
\rightarrow \KZ^{*}\circ \KZ(T_{c-\psi})$, where we write $\KZ^{*}$ for the right
adjoint functor of $\KZ$. Since the socle of $T_{c-\psi}$ does not contain 
simples annihilated by $\KZ$, we see that $T_{c-\psi}\subset\KZ^{*}\circ \KZ(T_{c-\psi})$. 
Note that $T_{c-\psi}\subset \operatorname{im}\iota$. Indeed, by Lemma \ref{Lem:tilt_basic}, either $\KZ^{*}\circ \KZ(T_{c-\psi})=T_{c-\psi}$ or $$\operatorname{codim}\VA(\KZ^{*}\circ \KZ(T_{c-\psi})/T_{c-\psi})=1$$ and, furthermore, the head of $T_{c-\psi}$ consists of simples
with associated variety $\h$. Now $T_{c-\psi}\subset \operatorname{im}\iota$ follows from
\begin{equation}\label{eq:loc_equal}
\mathsf{loc}(\operatorname{im}\iota)=\mathsf{loc}(T_{c-\psi})
\end{equation}

All simples in the head of $\mathcal{B}_c(\psi)\otimes_{H_c}P_c$ have associated variety
$\h$. Indeed, an epimorphism $\mathcal{B}_c(\psi)\otimes_{H_c}P_c\twoheadrightarrow
L$, where $L$ is a simple, gives rise to a nonzero homomorphism $\mathcal{B}_c(\psi)
\rightarrow L(P_c,L)$, where the target bimodule is the Harish-Chandra part of
$\operatorname{Hom}_\C(P_c,L)$ (this bimodule is HC by \cite[Section 5.7]{sraco}).
If $\VA(L)\neq \h$, then $\VA(L(P_c,L))\neq (\h\oplus \h^*)/W$.
This is impossible because $\mathcal{B}_c(\psi)$ is simple and $\VA(\mathcal{B}_c(\psi))=(\h\oplus\h^*)/W$.
This proves the claim in the beginning of the paragraph. Together with
(\ref{eq:loc_equal}) this implies
$\iota: \mathcal{B}_c(\psi)\otimes_{H_c}P_c\twoheadrightarrow T_{c-\psi}$.

By Lemma \ref{Lem:ring_rk1}, the derived tensor product with the bimodule $\A_{\lambda,\lambda^{opp}-\lambda}^{(\theta)}$
is the inverse Ringel duality. It follows that the object
${}^{\OCat}\Res_{W}^{W_H}(\mathcal{B}_c(\psi)\otimes_{H_c}P_c)\cong
\mathcal{B}_c(\psi)_{\dagger,W_H}\otimes_{H_c(W_H)}{}^{\OCat}\Res_{W}^{W_H}(P_c)$
is tilting so there are no finite dimensional modules in the
the socle of ${}^{\OCat}\Res_{W}^{W_H}(\mathcal{B}_c(\psi)\otimes_{H_c}P_c)$.
Together with (\ref{eq:loc_equal}) this implies
\begin{equation}\label{eq:res_vanish}
\Res_{W}^{W_H}(\ker\iota)=0, \forall H.
\end{equation}

The object $T_{c-\psi}$ has no extensions by simples with associated
variety of codimension more than $1$ in either direction, Lemma \ref{Lem:tilt_basic}.
So (\ref{eq:res_vanish}) implies that $\mathcal{B}_c(\psi)\otimes_{H_c}P_c=T_{c-\psi}\oplus \ker\iota$.
But we have already seen in this proof that $\operatorname{Hom}_{\mathcal{O}_{c-\psi}}(\mathcal{B}_c(\psi)\otimes_{H_c}P_c, \ker\iota)=0$. This finally implies that $\iota$ is an isomorphism.

Note that under this isomorphism,  the summand $\mathcal{B}_c(\psi)\otimes_{H_c}P_c(\lambda)$
coincides with $T_{c-\psi}(\lambda)$. This is because of Corollary \ref{Cor:loc_coinc}.

By the above, $\mathcal{B}_c(\psi)\otimes^L_{H_c}\bullet=\mathcal{B}_c(\psi)\otimes_{H_c}\bullet$
gives a functor  $\OCat_c(W)\operatorname{-proj}\rightarrow \OCat_{c-\psi}(W)\operatorname{-tilt}$.
Both this functor and $R^{-1}$ make the following diagram commutative.

\begin{picture}(90,30)
\put(2,2){$\KZ(\OCat_c\operatorname{-proj})$}
\put(5,22){$\OCat_c\operatorname{-proj}$}
\put(60,2){$\KZ(\OCat_{c-\psi}\operatorname{-tilt})$}
\put(63,22){$\OCat_{c-\psi}\operatorname{-tilt}$}
\put(9,20){\vector(0,-1){13}}
\put(67,20){\vector(0,-1){13}}
\put(22,23){\vector(1,0){39}}
\put(27,3){\vector(1,0){30}}
\put(39,4){$=$}
\end{picture}

Since $\KZ$ is fully faithful on both $\OCat_c\operatorname{-proj}, \OCat_{c-\psi}\operatorname{-tilt}$, we conclude
that $\mathcal{B}_c(\psi)\otimes^L_{H_c}\bullet\cong R^{-1}$.
\end{proof}

\begin{Cor}\label{Cor:Ringel}
Under the assumptions of Theorem \ref{Thm:Ring_dual_HC}, we have $\B_c(\psi)\otimes_{H_c}\Delta_c(\lambda)=\nabla_{c-\psi}(\lambda)$
and $\operatorname{Tor}^i_{H_c}(\B_c(\psi),\Delta_c(\lambda))=0$ for $i>0$.
\end{Cor}

\section{Derived equivalences}\label{S_derived}
\subsection{Scheme of proof}
In this section, we prove Theorem \ref{Thm:der_equiv}. The proof is basically in three steps.
The first one is Theorem \ref{Thm:Ring_dual_HC}.


Second, let $\param^1$ be a hyperplane in $\param$ and $\psi\in \underline{\param}_{\Z}$ be such that $c\in \param^1,
c-\psi$ lie in opposite open chambers provided $c$ is Weil generic (let us note that Weil generic points
of a  hyperplane define equal $c$-orders). We will produce a HC $H_{\param^1-\psi}$-$H_{\param^1}$ bimodule
$\mathcal{B}_{\param^1}(\psi)$ whose Weil generic fiber coincides with $\mathcal{B}_c(\psi)$. This will be done
in Section  \ref{SS_HC_fam}.

Third, Section \ref{SS_degen}, we will prove that, for a Zariski generic $c\in \param^1$, the functor $\mathcal{B}_c(\psi)\otimes^L_{H_c}\bullet$
is a derived equivalence $D^b(\OCat_c)\rightarrow D^b(\OCat_{c-\psi})$.

After these three steps, Theorem \ref{Thm:der_equiv} follows from Proposition \ref{Prop:Chered_equi_impr}
and several easy observations.
We will prove the theorem carefully in  Section \ref{SS_der_proof_compl}.
In Section \ref{SS_count} we will provide an application
of Theorem \ref{Thm:der_equiv} to counting the simple objects in $\OCat$ with given
associated variety.

\subsection{Family of HC bimodules}\label{SS_HC_fam}
Let $\param^1$ be an affine subspace in $\param$ and $\psi\in \underline{\param}_{\Z}$ be such that $c\in \param^1,
c-\psi$ lie in opposite open chambers provided $c$ is Weil generic in $\param^1$. Our goal is to produce a HC bimodule
$\B_{\param^1}(\psi)\in \HC(H_{\param^1},-\psi)$ with $\B_{\param^1}(\psi)_c=\B_c(\psi)$ for a Weil generic
$c\in \param^1$.

The idea is as follows. The bimodule $\B_{\param^1,\psi}$ still makes sense and its specialization
to $c\in \param^1$ is $\B_{c,\psi}$. We need to ``cut'' $\B_{\param^1,\psi}$ removing everything with
proper associated variety (for a Weil generic $c$) from the head and from the socle. We will see that
there is an ideal $\tilde{I}$ in $H_{\param^1}$ such that $H_c/\tilde{I}_c$ is the maximal quotient
with proper associated variety for Weil generic $c$. Then we cut ``small'' bimodules from the socle
by using the induction and restriction functors and from the head by multiplying by $\tilde{I}$. To construct
$\tilde{I}$ we first produce $I\subset H_{\param^1}$ such that $I_c\subset H_c$ is the minimal
ideal of finite codimension in $H_c$ for a Weil generic $c$.

\begin{Lem}\label{Lem:fin_codim_ideal}
Let $\param^1\subset\param$ be an affine subspace. There is a two-sided ideal $I\subset H_{\param^1}$ with the following
two properties:
\begin{itemize}
\item[(i)] $H_{\param^1}/I$ is a finitely generated  $\C[\param^1]$-module.
\item[(ii)] For a Weil generic $c\in \param^1$, the specialization $I_c:=I\otimes_{\C[\param^1]}\C_c$ is the minimal
ideal of finite codimension in $H_c$.
\end{itemize}
\end{Lem}
\begin{proof}
Consider the ideal $I(k)\subset H_{\param^1}$ generated by the elements of the form
\begin{equation}\label{eq:identity}\sum_{\sigma\in \mathfrak{S}_{2k}}\operatorname{sgn}(\sigma)a_{\sigma(1)}\ldots a_{\sigma(2k)}
\end{equation}
for arbitrary $a_1,\ldots,a_{2k}\in H_{\param^1}$. By the Amitsur-Levitski theorem, any $k$-dimensional
representation of $H_{\param^1}$ factors through $H_{\param^1}/I(k)$. Also it is clear from the definition
that $I(1)\supset I(2)\supset\ldots$. The quotients $I(k-1)/I(k)$ are HC $H_{\param^1}$-bimodules
(because any two-sided ideal in $H_{\param^1}$ is HC, and any quotient of a HC bimodule is again HC).
So their supports in $\param^1$ are constructible subsets, Lemma \ref{Lem:HC_supp}.

Now we claim that $\VA(H_{\param^1}/I(k))=\{0\}$. This is proved by analogy with the proof of \cite[Theorem 7.2.1]{miura}.
Namely, suppose that $b\in \h\setminus \{0\}$ is in $\VA(H_{\param^1}/I(k))$.
Consider the algebra $H_{\param^1}^{\wedge_b}$ and its two-sided ideal $I(k)^{\wedge_b}$.
This ideal is proper by the choice of $b$, and, on the other hand, $H_{\param^1}^{\wedge_b}/I(k)^{\wedge_b}$
satisfies the identity (\ref{eq:identity}). By results of Bezrukavnikov and Etingof
recalled in Section \ref{SS_ind_res_O}, we have a homomorphism $D(\h^{W_b})\rightarrow H_{\param^1}^{\wedge_b}$
and hence a homomorphism $D(\h^{W_b})\rightarrow H_{\param^1}^{\wedge_b}/I(k)^{\wedge_b}$. The algebra
$D(\h^{W_b})$ is simple and so this homomorphism is injective. Therefore $D(\h^{W_b})$
satisfies (\ref{eq:identity}), which is a contradiction. So, indeed,
$\VA(H_{\param^1}/I(k))=\{0\}$. This implies that $H_{\param^1}/I(k)$ is finitely
generated over $\C[\param^1]$.
In particular, for any $c$, the ideal $I_c(m)\subset H_c$ has finite codimension for any $m$.

We claim that only finitely many of
the supports of $I(k-1)/I(k)$   are dense in $\param^1$.
Indeed, assume the contrary: there is an infinite sequence $k_1<k_2<\ldots$
such that the support of $I(k_i-1)/I(k_i)$ is dense in $\param^1$.
Let $\param^1_i$ be a Zariski open subset of $\param^1$ with the properties that $I(k_i-1),I(k_i)$
are free over $\C[\param^1_i]$ and $I(k_i-1)/I(k_i)$ is free nonzero over
$\C[\param^1_i]$. The existence of $\param^1_i$ follows from Lemma \ref{Lem:HC_supp} and our assumption
on the support of $I(k_i-1)/I(k_i)$.

Take $c\in \cap_i \param^1_i$. Then  we have a sequence of ideals, $H_c\supset I_c(k_1)
\supsetneq I_c(k_2)\supsetneq\ldots$. Being a Serre subcategory in the category $\OCat$, the category $H_c\operatorname{-mod}_{fin}$ of finite dimensional $H_c$-modules has enough projectives (and there are finitely many of those). So $H_c/I_c(k_i)$ is the quotient of the direct sum of certain projectives in
$H_c\operatorname{-mod}_{fin}$. In particular, $I_c(k_i)$
contains the annihilator of the direct sum of all projectives in $H_c\operatorname{-mod}_{fin}$ that is an ideal of finite codimension. It follows that there is $j$ such that $I_c(k_i)=I_c(k_j)$ for all $i\geqslant j$. Therefore
the fiber of $I(k_{i-1})/I(k_i)$ at $c$ is zero for all $i>j$. This contradicts the choice of $c$.
So, for some $m$, the support of
$I(k-1)/I(k)$ for $k>m$ is contained in some proper Zariski closed subset of $\param^1$ (depending on $k$).

We set $I:=I(m)$. Condition (i) has been already established, while (ii) follows from the choice of $m$. 
\end{proof}

Now let us define a two-sided ideal $\tilde{I}\subset H_{\param^1}$. Let $W'$ be a parabolic subgroup
of $W$. For a two-sided ideal $J\subset H_{\param^1}(W')$ such that $H_{\param^1}(W')/J$
is finitely generated over  $\C[\param^1]$, define the ideal $J^{\dagger,H,W'}$
as the kernel of $H_{\param^1}\rightarrow (H_{\param^1}(W')/J)^{\dagger, W'}$.

Let $I(W')$ stand for the ideal in $H_{\param^1}(W')$ defined similarly to $I\subset H_{\param^1}$.
We set
$$\tilde{I}:=\left(\bigcap_{W'\neq \{1\}} I(W')^{\dagger,H,W'}\right)^n,$$
where the superscript means the $n$th power. Note that $H_{\param^1}/I(W')^{\dagger,H,W'}$
has proper associated variety for any $W'$. So  $H_{\param^1}/\tilde{I}$ has proper
associated variety as well.

\begin{Lem}\label{Lem:tilde_I_prop}
For a Weil generic $c$ and any HC $H_{c-\psi}$-$H_c$-bimodule $M$ with proper associated variety,
we have $M\tilde{I_c}=0$ and $\tilde{I}_c^2=\tilde{I}_c$.
\end{Lem}
\begin{proof}
Take $c$ so that (ii) of Lemma \ref{Lem:fin_codim_ideal} is satisfied for all possible $W'$
and such that $H_{\param^1}(W)/\tilde{I}, \tilde{I}$ are flat over some Zariski open neighborhood
of $c$ in $\param^1$.

Let $W_1,\ldots,W_\ell$ be the parabolic subgroups corresponding to the irreducible components
of $\VA(M)$. Consider the morphism $M\rightarrow \bigoplus_{i=1}^\ell(M_{\dagger,W_i})^{\dagger,W_i}$
and let $K$ denote its kernel. The $H_{c}(W_i)$-bimodule $M_{\dagger,W_i}$ is finite dimensional
and so is annihilated by $I_c(W_i)$. So $(M_{\dagger,W_i})^{\dagger,W_i}$ is annihilated by
$I_c(W_i)^{\dagger,H,W_i}$. It follows that $M(\bigcap_{i=1}^{\ell} I_c(W_i)^{\dagger,H,W_i})\subset K$.
Note that $\VA(K)\subset \VA(M)\setminus \bigcup_{i=1}^\ell \overline{\Leaf}_{W_i}$.
In particular, $\dim \VA(K)\leqslant \dim \VA(M)-2$. Now we  use the induction on
$m:=\dim \VA(M)/2$ to show that $$M\left(\bigcap_{W'\neq \{1\}} I_c(W')^{\dagger,H,W'}\right)^m=0.$$
In particular, $M\tilde{I}_c=0$.

The equality $\tilde{I}_c^2=\tilde{I}_c$ (for a Weil generic $c$) follows from the observation that $H_c/\tilde{I}_c^2$
has proper support.
\end{proof}

Now let $\param^1,\psi$ be as in the beginning of the section. Our goal is to produce a bimodule
$\B_{\param^1}(\psi)\in \HC(H_{\param^1},-\psi)$.
We start with the $H_{\param^1-\psi}$-$H_{\param^1}$-bimodule $\B_{\param^1,\psi}$.
Consider the natural homomorphism
\begin{equation}\label{eq:bimod_hom}
\B_{\param^1,\psi}\rightarrow (\B_{\param^1,\psi,\dagger,\{1\}})^{\dagger,\{1\}}.\end{equation}
Let $\hat{\B}_{\param^1,\psi}$ denote the image. Let $U\subset \param^1$ be  a non-empty Zariski open subset such that the
kernel, the image and the cokernel of (\ref{eq:bimod_hom}) are flat over $U$, the existence of $U$
follows from Lemma \ref{Lem:HC_supp}.

We claim that for $c\in U$,
the bimodule $\hat{\B}_{c,\psi}$ has no subbimodules with proper associated variety.
Since the functor $\bullet^{\dagger,\{1\}}$ is left exact, the specialization of
$(\B_{\param^1,\psi,\dagger,\{1\}})^{\dagger,\{1\}}$ at $c$ is naturally a submodule
in $(\B_{c, \psi,\dagger,\{1\}})^{\dagger,\{1\}}$.
Therefore
\begin{equation}\label{eq:bimod_incl}
\hat{\B}_{c,\psi}\subset (\B_{c, \psi,\dagger,\{1\}})^{\dagger,\{1\}}\end{equation}
Since $\bullet_{\dagger,\{1\}}$ kills all bimodules with proper associated variety
and $\bullet^{\dagger,\{1\}}$ is a right adjoint of $\bullet_{\dagger,\{1\}}$, we see
that the bimodule $(\B_{c, \psi,\dagger,\{1\}})^{\dagger,\{1\}}$ has no subbimodules
with proper associated variety. It  follows from (\ref{eq:bimod_incl}) that
$\hat{\B}_{c,\psi}$ has no subbimodules with  proper support.

We set $\B_{\param^1}(\psi):=\hat{\B}_{\param^1,\psi}\tilde{I}$. We claim that for $c$ as in Lemma
\ref{Lem:tilde_I_prop},
$\B_{\param^1}(\psi)_c$ has no quotients with proper associated variety. Indeed, by Lemma \ref{Lem:tilde_I_prop},
such a quotient of $\B_{\param^1}(\psi)_c$ has to be annihilated by $\tilde{I}_c$ that is impossible because
$\tilde{I}_c^2=\tilde{I}_c$. So we conclude that, for a Weil
generic $c$, the following holds:
\begin{itemize}
\item The specialization $\B_{\param^1}(\psi)_c$ has no submodules and quotients with proper associated variety.
\item $\B_{\param^1}(\psi)_c$ is a subquotient of $\B_{c,\psi}$.
\end{itemize}
But $\B_{c,\psi}$ has a unique composition factor with full associated variety and this factor is $\B_{c}(\psi)$.
The equality $\B_{\param^1}(\psi)_c=\B_c(\psi)$, for a Weil generic $c\in \param^1$, follows.

Below we write $\B_c(\psi)$ for $\B_{\param^1}(\psi)_c$ when $c$ is Zariski generic.

\subsection{Degeneration}\label{SS_degen}
Let  an affine subspace $\param^1$ in $\param$ and $\psi\in \underline{\param}_\Z$ be such that
\begin{itemize}
\item For a Weil generic $c\in \param^1$, the parameters $c$ and $c-\psi$ lie in opposite open chambers.
\item For a Zariski generic $c\in \param^1$, the parameters $c$ and $c-\psi$ are spherical.
\end{itemize}

Our goal in this subsection is to prove the following result.

\begin{Prop}\label{Prop:equi_Zariski}
There is a non-empty Zariski open subset $U\subset \param^1$ such that
$\mathcal{B}_{c}(\psi)\otimes^L_{H_c}\bullet: D^b(\OCat_c)\rightarrow D^b(\OCat_{c-\psi})$
is an equivalence of triangulated categories for any $c\in U$.
\end{Prop}
%
%
%
\begin{proof}
%
The proof is in several steps.

{\it Step 1}. Let $D^b_\OCat(H_?)$ (where $?$ means $c$ or $c'$)
denote the full subcategory in $D^b(H_?\operatorname{-mod})$
consisting of all complexes with homology in $\OCat_?$, this category is naturally identified
with $D^b(\OCat_?)$ by \cite[Proposition 4.4]{etingof_affine}. Also consider the category $D^b_{\HC}(H_{c'}\operatorname{-}H_c)$ of
all complexes of $H_{c'}$-$H_c$-bimodules with HC homology. Then $\bullet\otimes^L_{H_c}\bullet:
D^b(H_{c'}\operatorname{-}H_c)\times D^b(H_c)\rightarrow D^b(H_{c'})$ restricts to
$$D^b_{\HC}(H_{c'}\operatorname{-}H_c)\times D^b_\OCat(H_c)\rightarrow D^b_{\OCat}(H_{c'})$$
and a similar statement holds for $R\Hom_{H_c'}(\bullet,\bullet)$ thanks to Proposition \ref{Prop:HC_Ext}.

{\it Step 2}. Let $\mathcal{B}\in \HC(H_c,c'-c)$.
A right adjoint to $\mathcal{B}\otimes^L_{H_c}\bullet$
is given by $R\Hom_{H_{c'}}(\mathcal{B},\bullet)=R\Hom_{H_{c'}}(\mathcal{B},H_{c'})\otimes^L_{H_{c'}}\bullet$.
Note that  $R\Hom_{H_{c'}}(\mathcal{B}, \mathcal{B}\otimes^L_{H_c}\bullet)=R\End_{H_{c'}}(\B)\otimes^L \bullet$.

We claim that the following two claims are equivalent.
\begin{itemize}
\item[(a)] The functor $\mathcal{B}\otimes^L_{H_c}\bullet: D^b(\OCat_c)\rightarrow D^b(\OCat_{c'})$
is a category equivalence.
\item[(b)] The adjunction unit $H_c\rightarrow R\End_{H_{c'}}(\B)$ and the adjunction
counit $\B\otimes^L_{H_{c}}R\Hom_{H_{c'}}(\B, H_{c'})\rightarrow H_{c'}$ are isomorphisms.
\end{itemize}
It is clear that (b) implies (a). Let us prove that (a) implies (b).


We have (adjunction) isomorphisms of functors
$$\operatorname{Id}\rightarrow R\End_{H_{c'}}(\mathcal{B})\otimes^L_{H_c}\bullet,
\B\otimes^L_{H_{c'}}R\Hom_{H_{c'}}(\mathcal{B},H_{c'})\otimes^L_{H_{c'}}\bullet\rightarrow \operatorname{Id}.$$
Consider the cone $C$ of $H_c\rightarrow R\End_{H_{c'}}(\B)$, this is
an object in $D^b_{HC}(H_c\operatorname{-}H_c)$ since the category of HC bimodules is a Serre
subcategory in the category of all bimodules. We see that $C\otimes^L_{H_c}P_c=0$. Thanks to Lemma \ref{Lem:tens_HC_proj},
we have $H_i(C\otimes^L P_c)=H_i(C)\otimes_{H_c}P_c$.
By Lemma \ref{Lem:tens_HC_prgen}, $H_i(C)=0$.

The claim that $\B\otimes^L_{H_{c}}R\Hom_{H_{c'}}(\B, H_{c'})\xrightarrow{\sim} H_{c'}$ is proved similarly.

{\it Step 3}.
Let us check that (b) holds for $\B:=\B_{c}(\psi)$ and a Zariski generic $c$. Note that it holds for
a Weil generic $c\in \param^1$ because (a) holds there. Let us prove that
the counit morphism in (b) is an isomorphism for Zariski generic $c$. Let $C'_{\param^1}$ be the cone
of the natural homomorphism $\B_{\param^1}(\psi)\otimes^L_{H_{\param^1}}
R\Hom_{H_{\param^1-\psi}}(\B_{\param^1}(\psi),H_{\param^1-\psi})\rightarrow H_{\param^1-\psi}$ so that $C'_{\param^1}$
is an object in $D^b_{HC}(H_{\param^1-\psi}\operatorname{-}H_{\param^1}\operatorname{-bimod})$.
So $H_i(C'_{\param^1})$ is a HC bimodule. Hence it is generically free over $\C[\param^1]$.
Also note that $C'_{\param^1}\otimes^L_{\C[\param^1]}\C_c$ is the cone $C'_c$ of
$\B_{c}(\psi)\otimes^L_{H_{c'}}
R\Hom_{H_{c'}}(\B_{c}(\psi),H_{c'})\rightarrow H_{c'}$. So we conclude, that for a Zariski
generic $c$, we have that $H_i(C'_c)$ is the specialization of $H_i(C'_{\param^1})$ to $c$.
It follows that this specialization is zero for a Weil generic $c$. Hence it is zero
for a Zariski generic $c$ as well and $C'_c=0$.
\end{proof}

\begin{Rem}
We would like to remark that $\B_c(\psi)\otimes^L_{H_c}\Delta_c(\lambda)$ has no higher homology
and its class in $K_0$ coincides with that of $\nabla_{c-\psi}(\lambda)$, when $c$
is Zariski generic. The proof is similar to that of Step 3 and is based on Corollary
\ref{Cor:Ringel}.
\end{Rem}

\subsection{Proof of the main result}\label{SS_der_proof_compl}
Now we are ready to prove Theorem \ref{Thm:der_equiv}. Pick a parameter $c$. The order $\leqslant^c$
is refined by $\leqslant^{\tilde{c}}$ for $\tilde{c}\in c+\underline{\param}_{\Z}$ lying in an open chamber.
Thanks to  Proposition \ref{Prop:Chered_equi_impr}, we may replace $c$ with $\tilde{c}$ without changing
the abelian category and assume that $c$ is Zariski generic and hence lies in an open chamber $\mathfrak{C}$. Also it is enough to establish a
derived equivalence in the case when $c'$ lies in a chamber $\mathfrak{C}'$
that shares a wall $\Pi_0$ with  $\mathfrak{C}$ (in the general case, we take the composition of a sequence of equivalences,
each crossing a single wall). Next, we may assume that $c'-c\in \underline{\param}_{\Z}$. Indeed, otherwise
we can modify $c$ by subtracting an element  $\psi'\in \param_{\Z}$ such
that a bimodule $\B_{c,\psi'}$ with $c'-c+\psi'\in \underline{\param}_{\Z}$ is a Morita equivalence
and $c-\psi'\in \mathfrak{C}$, this follows from Corollary \ref{Cor:spherical} and the construction
of $\B_{c,\psi'}$.

Then replacing both $c,c'$ with points of $(c+\underline{\param}_{\Z})\cap \mathfrak{C}, (c'+\underline{\param}_{\Z})
\cap \mathfrak{C}'$  (so that the categories $\OCat_c,\OCat_{c'}$ stay the
same by Proposition \ref{Prop:Chered_equi_impr})
we may assume that $c,c'$ lie on hyperplanes $\Pi,\Pi'$ parallel to the wall $\Pi_0$
separating $\mathfrak{C},\mathfrak{C}'$ such that
Zariski generic parameters in $\Pi,\Pi'$ are spherical.
Now, by Proposition \ref{Prop:equi_Zariski}, we can modify $c,c'$ by the same element of $\Pi_0\cap \underline{\param}_{\Z}$ (this intersection is
a lattice in $\Pi_0$ because $\Pi_0$ is defined over $\Q$) staying in the same chambers
so that $\B_c(\psi)\otimes^L_{H_c}\bullet$ is an equivalence.

Finally, we need to show that the equivalence $\B_c(\psi)\otimes^L_{H_c}\bullet$ intertwines the KZ
 functors. It follows from the construction of $\B_c(\psi)$ that $\B_c(\psi)[\delta^{-1}]=\B_{c,\psi}[\delta^{-1}]=D(\h^{reg})\#W$. So $\B_c(\psi)\otimes^L_{H_c}\bullet$ intertwines
the functors $\loc$ and hence $\KZ$.  This completes the proof of Theorem \ref{Thm:der_equiv}.

\subsection{Application to counting}\label{SS_count}
The associated variety of a simple in $\OCat_c$ coincides with $W\h^{W'}$ for some
parabolic subgroup $W'\subset W$, see \cite[Section 3.8]{BE}. Let $n_{W'}(c)$ denote the number of
simples with associated variety $W\h^{W'}$.

\begin{Prop}\label{Prop:simple_count}
Let $\psi\in \param_{\Z}$. Then $n_{W'}(c)=n_{W'}(c-\psi)$ for all $c$.
\end{Prop}
\begin{proof}
If $\psi=-\bar{\chi}$ and $\B_{c,\bar{\chi}}$ is a Morita equivalence, the claim is clear. When $c,c':=c-\psi$
satisfy the assumptions of Proposition \ref{Prop:Chered_equi_impr}, the equivalence of that proposition preserves
the supports, see  \cite[6.4.9]{GL}. So we can assume that $\B_c(\psi)\otimes^L_{H_c}\bullet:D^b(\OCat_c)\rightarrow
D^b(\OCat_{c-\psi})$ is an equivalence. Let $D^b_{W'}(\OCat_c)$ denote the full subcategory of $D^b(\OCat_c)$
consisting of all complexes with homology having associated variety inside $W\h^{W'}$. From the compatibility
of the restriction functors and Tor's (see the end of Section \ref{SS_Tor_Ext})
it follows that $\B_c(\psi)\otimes^L_{H_c}\bullet$ maps $D^b_{W'}(\OCat_c)$ to $D^b_{W'}(\OCat_{c-\psi})$. A quasi-inverse
functor is $R\operatorname{Hom}_{H_{c'}}(\B_c(\psi),H_{c'})\otimes^L_{H_{c'}}\bullet$
so it maps $D^b_{W'}(\OCat_{c-\psi})$ to $D^b_{W'}(\OCat_c)$. We conclude that the categories
$D^b_{W'}(\OCat_{c-\psi})$ and $D^b_{W'}(\OCat_c)$ are equivalent that implies $n_{W'}(c)=n_{W'}(c-\psi)$.
\end{proof}

\section{Perverse equivalences}\label{S_perv}
\subsection{Main result}\label{SS_main_perv}
In this section we are going to prove that there are perverse equivalences between some categories
$D^b(\OCat_c),D^b(\OCat_{c'})$. Namely, suppose an affine subspace $\param^1\subset\param$
and $\psi\in \underline{\param}_{\Z}$ are such that
\begin{enumerate}
\item Both $\param^1, \param^1-\psi$ contain Zariski open subsets of spherical elements.
\item For a Weil generic $c\in \param^1$, the parameters $c,c-\psi$ lie in opposite
open chambers.
\end{enumerate}
Then, for a Zariski generic $c$, the functor $\varphi_c:=\B_c(\psi)\otimes^L_{H_c}\bullet:D^b(\OCat_c)\rightarrow D^b(\OCat_{c'})$
is an equivalence of triangulated categories. We are going to show that (possibly after restricting
to a smaller Zariski open subset)  the equivalence $\varphi_c$ is perverse. This is an analog of
\cite[Theorem 7.2]{BL}.

The corresponding filtrations are produced similarly to \cite[Section 7]{BL}. Namely, recall the ideals
$I(W')^{\dagger,H,W'}\subset H_{\param^1}(W)$ defined before Lemma \ref{Lem:tilde_I_prop}.
Define the ideal $\J_k\subset H_{\param^1}$ as follows:
$$\J_k:=\left(\bigcap_{W'} I(W')^{\dagger,H,W'}\right)^{k},$$
where the intersection is taken over all parabolic subgroups $W'$ with $\dim \h^{W'}\leqslant k-1$.
The ideal $\J_k$ has the following important property. Let $c$ be a Weil generic element of $\param^1$,
then the specialization $\J_{k,c}$ coincides with the minimal ideal $J\subset H_c$ such that
$\dim \VA(H_c/J)<2k$, this is proved analogously to Lemma \ref{Lem:tilde_I_prop}. 
Note that, in this case, $\J_{k,c}^2=\J_{k,c}$, in particular,
the subcategory in $\OCat_c$ of all modules annihilated by $\J_{k,c}$ is closed under extensions
and hence is a Serre subcategory. Also note that $H_c=\J_{0,c}\supset\J_{1,c}\supset\J_{2,c}\supset\ldots
\supset \J_{n,c}\supset \J_{n+1,c}=\{0\}$ for all parameters $c$. Let $\Cat^1=\OCat_c$ and $\Cat^1_j$ be the full subcategory in $\OCat_c$
consisting of all modules annihilated by $\J_{n+1-j,c}$. So we get a filtration of $\Cat^1$
by Serre subcategories.

Define  ideals $\J'_i\subset H_{\param^1-\psi}$ and subcategories $\Cat^2_j\subset \Cat^2:=\OCat_{c-\psi}$
in a similar way.

\begin{Thm}\label{Thm:perv}
There is a non-empty Zariski open subset $U\subset \param^1$ such that $\J_{k,c}^2=\J_{k,c}$
for all $k$ and the equivalence $\B_c(\psi)\otimes^L_{H_c}\bullet$ is perverse with respect
to the filtrations introduced above.
\end{Thm}

If $c,c'$ lie in the opposite chambers, Theorem \ref{Thm:perv} follows  from Lemma \ref{Lem:D_perv}.
Indeed, the filtrations on $\OCat_c,\OCat_{c'}$ are the filtrations by dimensions of support as in Section
\ref{SS_Ringel}. This is a consequence of \cite[Theorems 1.2,1.3]{B_ineq}. Recall that the equivalence $\OCat_{c}^{opp,r}\cong \OCat_{c'}$ preserves supports, see
the end of Section \ref{SS_KZ_Ringel}. This implies Theorem \ref{Thm:perv} in this case.

In general, we will, roughly speaking, show that in our situation the perversity is preserved under degeneration.
Note that since $\J_{k,c}^2=\J_{k,c}$ for a Weil generic $c$, this equality also holds for a Zariski generic $c$.

Below we will write $H'_{\param^1}$ for $H_{\param^1-\psi}$ and $H'_c$ for $H_{c-\psi}$.
Further, we set $\B_c:=\B_c(\psi)$.

\subsection{Computations with HC bimodules}
Here we are going to study various Tor's and Hom's between HC bimodules.

\begin{Prop}
For $c$ Zariski generic in $\param^1$, the following holds:
\begin{itemize}
\item[(a)] For all $i,j$, we have $\J'_{j,c}\operatorname{Tor}^{H_c}_i(\B_c, H_c/\J_{j,c})=0$.
\item[(b)] For all $i,j$, we have $\operatorname{Tor}^{H'_c}_i(H'_c/\J'_{j,c}, \B_c)\J_{j,c}=0$.
\item[(c)] We have   $\operatorname{Tor}^{H_c}_i(\B_c, H_c/\J_{j,c})=0$
  for $i<n+1-j$.
\item[(d)] We have $\J'_{j-1,c}\operatorname{Tor}^{H_c}_i(\B_c, H_c/\J_{j,c})=
\operatorname{Tor}^{H_c'}_i(H'_c/\J'_{j,c},\B_c)\J_{j-1,c}=0$
for $i>n+1-j$.
\item[(e)]  Set $\B_{j,c}:=\operatorname{Tor}^{H_c}_{n+1-j}(\B_c, H_c/\J_{j,c})$.
The kernel and the cokernel of the natural homomorphism $$\B_{j,c}\otimes_{H_c}
\operatorname{Hom}_{H'_c}(\B_{j,c}, H'_c/\J'_{j,c})\rightarrow H'_c/\J'_{j,c}$$
are annihilated by $\J'_{j-1,c}$ on the left and on the right.
\item[(f)] The kernel and the cokernel of the natural homomorphism
$$\operatorname{Hom}_{H_{c}}(\B_{j,c}, H_{c}/\J_{j,c})\otimes_{H'_c}\B_{j,c}
\rightarrow H_{c}/\J_{j,c}$$ 
are annihilated on the left and on the right by $\J_{j-1,c}$.
\end{itemize}
\end{Prop}
\begin{proof}
The proof is in four steps. First, we prove (a),(b) for an arbitrary Weil
generic $c$.
Then we will check (c)-(f) for $j=1$ and a Weil generic
$c$. In Step 3 we will establish these four claims with an arbitrary $j$ and a Weil generic $c$. Finally,  we will prove
the claims (a)-(f) for an arbitrary $j$ and a Zariski generic $c$.
Recall that, for a Weil generic $c$, the functor $\B_c\otimes^L_{H_c}\bullet:D^b(\OCat_c)
\rightarrow D^b(\OCat_{c'})$ is a perverse equivalence with respect to the filtrations
by the dimensions of support.

{\it Step 1}. Here $c$ is Weil generic.  In order to establish (a),(b), notice that all Tor's involved have GK dimension
$<2j$. (a) and (b) follow from a straightforward analog of Lemma \ref{Lem:tilde_I_prop}
(that is proved in the same way).

{\it Step 2}. In this step $c$ is again Weil generic.  It follows that $\J_{1,c},\J'_{1,c}$ are the minimal ideals of finite codimension in $H_c, H'_c$.

Note that by the perversity of $\B_c\otimes^L_{H_c}\bullet$, we have $\B_c\otimes^L_{H_c}H_c/\J_{1,c}=
\B_{1,c}[-n]$ that implies  (c),(d).  Further, the bimodule $\B_{1,c}$ defines a Morita equivalence between $H_{c}/\J_{1,c}$ and $H'_c/\J_{1,c}'$ that immediately implies (e) and (f).

%

{\it Step 3}. Before proving (c)-(f) for a Weil generic $c$,
we need some preparation. For a parabolic subgroup $W'\subset W$,
let $\B_c(W')$ be an analog of $\B_c$ for $W'$. Recall, Proposition \ref{Lem:B_socle}, that $\B_c(W')=(\B_c)_{\dagger,W'}$.
Also let $\J_{j,c}(W'),\J'_{j,c}(W')$ denote the ideals defined analogously to $\J_{j,c},\J'_{j,c}$.
We will need the following lemma.

\begin{Lem}\label{Lem:id_restr}
Assume $c$ is Weil generic in $\param^1$.  Let $W'$ be a parabolic subgroup with
$\dim \h^{W'}=j-1$. Then $(\J_{j,c})_{\dagger,W'}=\J_{1,c}(W')$
and similarly, $(\J'_{j,c})_{\dagger,W'}=\J'_{1,c}(W')$.
\end{Lem}
\begin{proof}[Proof of Lemma \ref{Lem:id_restr}]
The ideal $\J_{1,c}(W')\subset H_c(W')$ is the minimal ideal of finite codimension.
So we have an inclusion $(\J_{j,c})_{\dagger,W'}\supset\J_{1,c}(W')$. On
the other hand, by the construction, $\J_{j,c}\subset \J_{1,c}(W')^{\dagger,H,W'}$.
Therefore $(\J_{j,c})_{\dagger,W'}\subset (\J_{1,c}(W')^{\dagger,H,W'})_{\dagger,W'}$.
The construction of the functors $\bullet_{\dagger,W'},\bullet^{\dagger,W'}$
easily implies that $(\J^{\dagger,H,W'})_{\dagger,W'}\subset \J$ for any ideal $\J$.
\end{proof}

Finally, we will use Lemma \ref{Lem:dag_Tor} saying that the functor $\bullet_{\dagger,W'}$ intertwines
Tor's and Ext's.

Let us prove (c)-(f). Pick  $W'$ with $\dim \h^{W'}=j-1$. We have
$$\operatorname{Tor}_i^{H_c}(\B_c,H_c/\J_{j,c})_{\dagger,W'}=\operatorname{Tor}_i^{H_c(W')}(\B_c(W'), H_c(W')/\J_{1,c}(W')).$$ In particular, if $i>n+1-j$, the right hand side vanishes.
So $\operatorname{Tor}_i^{H_c}(\B_c,H_c/\J_{j,c})$ has GK dimension less than $2(j-1)$,
which implies (d). The proofs of (e) and (f) are similar (here we also need to use that
$\bullet_{\dagger,W'}$ intertwines Hom's and observe that applying $\bullet_{\dagger,W'}$
to the natural homomorphisms we get similar natural homomorphisms but for $W'$).

Let us prove (c): $\operatorname{Tor}_i^{H_c}(\B_c,H_c/\J_{j,c})=0$ for $i<n+1-j$.
This is done by induction on $j$, the base $j=1$ was established in the previous step.
Arguing as in the proof of (d), we see that the GK dimension of $\operatorname{Tor}_i^{H_c}(\B_c,H_c/\J_{j,c})$
is less than $2(j-1)$. Let us pick the minimal $i$ such that $\operatorname{Tor}_i^{H_c}(\B_c,H_c/\J_{j,c})$
is nonzero and let $\dim \VA(\operatorname{Tor}_i^{H_c}(\B_c,H_c/\J_{j,c}))=2(k-1)$ so that $k<j$. 
Consider the derived tensor product $\B_c\otimes^L_{H_c} H_c/\J_{k,c}=(\B_c\otimes^L_{H_c} H_c/\J_{j,c})\otimes^L_{H_c/\J_{j,c}}H_c/\J_{k,c}$.
We see that  the $i$th homology of the left hand side is zero, while on the right hand
side we have $\operatorname{Tor}_i^{H_c}(\B_c,H_c/\J_{j,c})\otimes_{H_c/\J_{j,c}}H_c/\J_{k,c}=
\operatorname{Tor}_i^{H_c}(\B_c,H_c/\J_{j,c})$, a contradiction.

{\it Step 4}. Now assume $c$ is Zariski generic.  Let us prove (a), the other 5 parts are proved
similarly. Consider the HC bimodules $\B_{\param^1}, H_{\param^1}/\J_{j,\param}$. They are 
generically flat over $\param^1$, see Lemma \ref{Lem:HC_supp},
and their Zariski generic specializations coincide with $\B_c,H_c/\J_{j,c}$. Then $\operatorname{Tor}^{H_{\param^1}}_i(\B_{\param^1},H_{\param^1}/\J_{j,\param^1})$ is again a 
HC bimodule and so is generically flat. Its Zariski  generic specialization 
coincides with $\operatorname{Tor}^{H_{c}}_i(\B_{c}, H_{c}/\J_{j,c})$. It follows that the support of
$\operatorname{Tor}^{H_{\param^1}}_i(\B_{\param^1}, H_{\param^1}/\J_{j,\param^1})$ is not Zariski  
dense. (a) for a Zariski generic $c$ follows.
\end{proof}

\subsection{Proof of Theorem \ref{Thm:perv}}
Let us prove (I) and (II) in the definition of a perverse equivalence, Section \ref{SS_perv_intro}.
Note that, for $M\in H_c/\J_{j,c}\operatorname{-mod},
N\in H'_c/\J'_{j,c}\operatorname{-mod}$, we have
\begin{equation}\label{eq:tensor_subcat}
\B_c\otimes^L_{H_c}M=(\B_c\otimes^L_{H_c}H_c/\J_{j,c})\otimes^L_{H_c/\J_{j,c}}M.
\end{equation}
\begin{equation}\label{eq:RHom_subcat}
R\operatorname{\Hom}_{H'_c}(\B_c,N)=R\Hom_{H'_c/\J_{j,c}'}(H'_c/\J_{j,c}'\otimes^L_{H'_c}\B_c,N).
\end{equation}
The inclusion $\mathcal{F}(\Cat^1_{n+1-j})\subset \Cat^2_{n+1-j}$ follows from (\ref{eq:tensor_subcat})
and  (a). The inclusion $\mathcal{F}^{-1}(\Cat^2_{n+1-j})\subset \Cat^1_{n+1-j}$
follows from (b) and (\ref{eq:RHom_subcat}).

(II) follows from (c) and (\ref{eq:tensor_subcat}).

Let us prove (III).
Conditions (e),(f) imply that
$\B_{c,j}\otimes_{H_{c}/\J_{c,j}}\bullet$ defines an equivalence
$\Cat^1_{n+1-j}/\Cat^1_{n-j}\xrightarrow{\sim} \Cat^2_{n+1-j}/\Cat^2_{n-j}$
of abelian categories. A formal corollary of this is that
$\operatorname{Tor}^i_{H_c/\J_{c,j}}(\B_{c,j},M)\in \Cat^2_{n-j}$
for all $M\in \Cat^1_{n+1-j}$ and all $i>0$. (d) and (\ref{eq:tensor_subcat})
finish the proof of (III).

\section{Open problems}\label{S_open}
\subsection{Extension to general SRA}
We have constructed derived equivalences between categories $\mathcal{O}$. One can ask to construct such equivalences
between the categories of all modules. We conjecture that $D^b(H_c\operatorname{-mod})\xrightarrow{\sim}
D^b(H_{c-\psi}\operatorname{-mod})$ for any $\psi\in \param_{\Z}$ (the proof of Proposition \ref{Prop:equi_Zariski}
shows that, for a fixed $\psi$,  this is the case when $c$ is Zariski generic). This conjecture is true when $W=G(\ell,1,n)$, this was shown in \cite[Section 5]{GL}. Moreover, in \cite[Section 5]{GL} the conjecture was also proved for the Symplectic reflection algebras corresponding to the groups $\Gamma_n:=\mathfrak{S}_n\ltimes \Gamma_1^n$,
where $\Gamma_1$ is a finite
subgroup of $\operatorname{SL}_2(\C)$. Now let $\Gamma\subset \operatorname{Sp}(V)$
be an arbitrary symplectic reflection group.
We conjecture that $D^b(H_c\operatorname{-mod})\xrightarrow{\sim}
D^b(H_{c-\psi}\operatorname{-mod})$ for any $c\in \param$ and any $\psi$ in a lattice $\param_{\Z}$ defined
in this case as follows. The parameter space $\param$ is the direct sum $\bigoplus_{i=1}^r \param^i$, where
$r$ is the number of the conjugacy classes of subspaces of the form $V^s\subset V$ of codimension $2$
and $\param^i$ is the parameter space for the pointwise stabilizer $\Gamma^i$ of $V^s$ (acting on $V/V^s\cong \C^2$).
Since $\Gamma^i$ is a Kleinian group, we have a lattice $\param^i_{\Z}\subset \param^i$ used
in \cite[Section 5]{GL}. We set $\param_{\Z}:=
\bigoplus_{i=1}^r \param^i_{\Z}$. This definition is compatible both with the Cherednik case
and with the case of $\Gamma_n$.


\subsection{Wall-crossing bijections}
Let $\varphi_c: D^b(\OCat_c)\xrightarrow{\sim}D^b(\OCat_{c-\psi})$ be a perverse equivalence constructed above.
As any perverse equivalence, it induces a bijection $\operatorname{Irr}(\OCat_c)\rightarrow \operatorname{Irr}(\OCat_{c-\psi})$. The question is to compute this bijection (a generalization
of the Mullineux involution). The case of the groups $G(\ell,1,n)$ is addressed in \cite{Cher_supp}.

\end{document}